\documentclass[reqno]{amsart}

\usepackage[citebordercolor={0.7 0.7 0.7},linkbordercolor={0.7 0.7 1}]{hyperref}
\usepackage{comment,graphicx,color}
\usepackage[all]{xy}
\excludecomment{maximacode}
\excludecomment{arxivabstract}
\includecomment{arma}
\usepackage{ltbunicode}

\ifcsname showkeystrue\endcsname
\usepackage[notcite,notref]{showkeys}

\fi

\DeclareGraphicsExtensions{.png,jpg,jpeg,.mps,.pdfmac,.pdffig,.spdf,.pdf}
\newcommand{\T}{{\bf T}}
\newcommand{\R}{{\bf R}}

\newcommand{\Z}{{\bf Z}}

\newcommand{\set}[1]{\left\{#1\right\}}

\newtheorem{theorem}{Theorem}[section]
\newtheorem{corollary}{Corollary}[section]

\newtheorem{lemma}{Lemma}[section]

\newtheorem{claim}{Claim}[theorem]
\theoremstyle{definition}
\newtheorem{definition}{Definition}[section]
\theoremstyle{remark}

\newtheorem{remark}{Remark}[section]

\newcommand{\safeinput}[1]{\IfFileExists{#1}{{\catcode`\&=0\input{#1}}}{\message{File
        #1 does not exists. Skipping.}}}
\def\ltxfigure#1#2#3{\resizebox{#2}{#3}{\input{#1}}}

\newcommand{\sign}[1]{{\mathrm{sign}}(#1)}
\newcommand{\degree}[1]{{\mathrm{deg}}(#1)}
\newcommand{\D}[1]{\mathrm{d}#1}
\newcommand{\mean}[1]{\overline{#1}}
\newcommand{\bmean}[1]{\hat{#1}}
\newcommand{\textem}[1]{{\em #1}}
\newcommand{\defn}[1]{\textem{#1}}

\ifcsname nhpaper\endcsname
\else
\fi

\newcommand{\divergence}[1]{{\mathrm{div}}(#1)}
\newcommand{\temp}[1]{2{\mathrm K}}

\newcommand{\kp}[2]{\langle #1, #2 \rangle}
\newcommand{\ddt}[2][\ ]{\frac{\D{#1}}{\D{#2}}}
\newcommand{\didi}[2][\ ]{\frac{∂#1}{∂#2}}
\newcommand{\cotangent}{T^*}
\newcommand{\metaref}[3]{{\ifnum#1=0(\fi}{#3}\ref{#2}{\ifnum#1=0)\fi}}
\renewcommand{\eqref}[2][0]{\metaref{#1}{#2}{eq.~}}
\newcommand{\itref}[2][0]{\metaref{#1}{#2}{}}
\newcommand{\ineqref}[2][0]{\metaref{#1}{#2}{ineq.~}}
\newcommand{\secref}[2][0]{\metaref{#1}{#2}{\S~}}

\newcommand{\pb}[2]{\left\{ #1,#2 \right\}}

\DeclareMathOperator{\hessian}{hess}
\DeclareMathOperator{\erfc}{erfc}

\newcommand{\gibbsexp}[2]{\mathbf{E}_{#1}[#2]}

\long\def\aureply#1{\ifhmode\newline\fi\noindent{\bf Author Reply}:\ {\em #1}}

\def\gettimestamp#1<#2>{\def\timestamp{\tt #2}}

\gettimestamp Time-stamp: <2019-08-03 09:17:30 CDT>

\begin{document}

\title[Invariant tori]{Invariant tori for a class of singly thermostated hamiltonians}
\author{Leo T. Butler}
\address{Department of Mathematics, University of Manitoba, Winnipeg,
  MB, Canada, R2J 2N2}
\email{leo.butler@umanitoba.ca}
\date{\timestamp}
\thanks{The author thanks W.~G. Hoover for his comments on a draft of
  this paper. Partially supported by the Natural Science and
  Engineering Research Council of Canada grant 320 852.}

\begin{abstract}
  This paper demonstrates sufficient conditions for the existence of a
  positive measure set of invariant KAM tori in a singly thermostated,
  1 degree-of-freedom hamiltonian vector field. This result is applied
  to 4 important single thermostats in the literature and it is
  shown that in each case, if the hamiltonian is real-analytic and
  well-behaved, then the thermostated system always has a positive
  measure set of invariant KAM tori for sufficiently weak coupling and
  high temperature. This extends results of Legoll, Luskin \&
  Moeckel~\cite{MR2299758,MR2519685}.
\end{abstract}
\begin{arxivabstract}
  This paper demonstrates sufficient conditions for the existence of a
  positive measure set of invariant KAM tori in a singly thermostated,
  1 degree-of-freedom hamiltonian vector field. This result is applied
  to 4 important single thermostats in the literature and it is
  shown that in each case, if the hamiltonian is real-analytic and
  well-behaved, then the thermostated system always has a positive
  measure set of invariant KAM tori for sufficiently weak coupling and
  high temperature. This extends results of Legoll, Luskin \&
  Moeckel.
\end{arxivabstract}

\maketitle

\section{Introduction} \label{sec:intro}

In equilibrium statistical mechanics, a mechanical hamiltonian $H$ is
viewed as the internal energy of an infinitesimal system that is
immersed in, and in equilibrium with, a heat bath $B$ at the
temperature $T$. A dynamical model of the exchange of
energy was introduced by Nos{é} \cite{nose}, based on earlier work of
Andersen \cite{andersen}. This consists of adding an extra degree of
freedom $s$ and rescaling momentum by $s$:
\begin{align}
  \label{eq:nose}
  F &= H(q,p s^{-1}) + N(s,p_s), &&& \text{where }N(s,p_s) =\dfrac{1}{2 M} p_s^2 + n k T \ln s,
\end{align}
$n$ is the number of degrees of freedom of $H$, $M$ is the
``mass'' of the thermostat and $k$ is Boltzmann's constant. Solutions
to Hamilton's equations for $F$ model the evolution of the state of
the infinitesimal system along with the exchange of energy with the
heat bath.

Nos{é}'s thermostat is reduced by eliminating the state variable $s$
and rescaling time $t$~\cite{hoover}. The Nos{é}-Hoover thermostat of
the $1$ degree of freedom hamiltonian $H$ is the following vector
field:
\begin{align}
  \label{eq:nose-hoover}
  \dot{q} &= H_{ρ}, && \dot{ρ} = -H_q - ε \xi ρ, &&
  \dot{\xi} = ε \left( ρ · H_{ρ} - T \right),
\end{align}
where $ε ² = 1/M$ and $ρ = p s^{-1}$.

Hoover observed that this thermostat was ineffective in producing the
statistics of the Gibbs-Boltzmann distribution from single orbits of
the thermostated harmonic oscillator~\cite{hoover}. Numerous
extensions of the Nos{é}-Hoover thermostat have appeared, but this
thermostat remains the touchstone in the literature. This note focuses
on those thermostats which model the exchange of energy with the heat
bath using a single, additional thermostat variable ($ξ$ in
\ref{eq:nose-hoover}), the so-called \textem{single thermostats}.

To state the main result of this paper, some terminology is
necessary. Let $Σ=\R$ or $\T^1$ and $H : \cotangent Σ → \R$ be a $C^r$
hamiltonian function. A \defn{thermostat} for $H$ is a vector field
$Τ=Τ_{β}$ on the extended phase space $P = \cotangent Σ × \R$ that
preserves a Gibbs-Boltzmann-like measure and ``heats''
(resp. ``cools'') $H$ at low (resp. high) temperature relative to the
temperature $T=1/β$. Definition~\ref{def:thermostat} makes precise
this heuristic description. The vector field
$Y_{ε,β} = X_H + ε\, Τ_{β}$ on the extended phase space $P$ describes
the thermostated hamiltonian system; the parameter $ε$ determines the
coupling strength between the system and heat bath. When $Τ$ is the
Nos{é}-Hoover thermostat, the vector field $Y_{ε,β}$ is described by
the differential equations~\eqref{eq:nose-hoover}.

A hamiltonian $H$ is \defn{well-behaved} if it satisfies
definition~\ref{def:h-is-well-behaved}. Roughly, this means that $H$
should be topologically well-behaved (proper, Morse, compact critical
set) and physically well-behaved (no local maxima, finite moments of
momentum at all positive temperatures).

\begin{theorem}
  \label{thm:intro-thm-1}
  Let $H : \cotangent Σ → \R$ be a real-analytic, well-behaved
  hamiltonian. If the thermostat vector field $Τ=Τ_{β}$ is a
  \begin{enumerate}
  \item Nos{é}-Hoover thermostat~\secref{sec:hoover-separable};
  \item Tapias, Bravetti \& Sanders logistic thermostat~\secref{sec:tap-brav-san};
  \item Watanabe-Kobayashi thermostat~\secref{sec:wat-kob}; or
  \item Hoover-Sprott-Hoover thermostat~\secref{sec:hs-hsh},
  \end{enumerate}
  then there is a function $E : (0,∞) → [0,∞)$, with finitely many
  zeros, such that for each $β>0$ and all $ε ∈ [0,E(β))$, the
  thermostated vector field $Y_{ε,β}$ on $P=\cotangent Σ × \R$ has a
  positive measure set of invariant tori.
\end{theorem}

The zeroes of $E$ are the inverse temperatures of critical cycles of
$H$. Except for case (3), the critical cycles have an effective
temperature of zero so $E$ is never zero; in case (3), a saddle may
have a positive effective temperature and for such (inverse)
temperatures $E=0$, i.e. the present techniques are not applicable.

Legoll, Luskin and Moeckel~\cite{MR2299758,MR2519685} prove the
existence of KAM tori for the Nos{é}-Hoover thermostated harmonic
oscillator; they indicate that the general case reduces to proving the
non-isochronicity of an associated averaged hamiltonian. The present
note derives the averaged hamiltonian and uses this to prove case (1)
of the above theorem~(see~\secref[1]{sec:hoover-separable} and
remark~\ref{rem:comments-on-hoovers-thermostat-in-lit}). Watanabe and
Kobayashi~\cite{PhysRevE.75.040102} introduce a $2$-parameter family
of thermostats. They show, with the thermostated harmonic oscillator,
that the associated averaged thermostat has a first integral. The
present note extends this by deriving the hamiltonian and symplectic
structure of the averaged thermostat in the general setting in order
to prove case (3) above. The thermostats of Tapias, Bravetti \&
Sanders and Hoover, Sprott \& Hoover have been investigated by
computationally-oriented researchers with the aim of finding ergodic
single
thermostats~\cite{10.12921/cmst.2016.0000061,10.12921/cmst.2017.0000005,
  10.1016/j.cnsns.2015.08.02}.

Theorem~\ref{thm:intro-thm-1} is proven in three parts: first, it is
proven that there is a cross-section such that the Poincar{é} return
map of the averaged system has a stable fixed point; second, if the
return map is isochronous in a neighbourhood of this fixed point,
real-analyticity globalizes the isochronicity; finally, a sufficiently
explicit hamiltonian of the return map is computed and used to prove
that the hamiltonian $H$ has a degenerate global minimum. Under much
weaker hypotheses (e.g. $H$ and $Τ$ are only $C^r$), one proves the
existence of positive-measure sets of invariant tori in a
neighbourhood of a thermostatic equilibrium which satisfies a twist
condition (Theorem~\ref{thm:main-thm}).

A related obstruction is observed in the classical adiabatic piston
problem. In one variant of this problem, a box of fixed size is filled
with a gas that is separated by a massive piston. The piston is free
to move parallel to an axis of the box without friction. Neishtadt \&
Sinai and Wright~\cite{Neishtadt2004,10.1007/s00220-007-0317-0} show
that the Anosov-Kasuga averaging theorem~\cite{}, combined with
ergodicity of the gas dynamics, imply that in the infinite mass limit
the piston oscillates deterministicly and for large but finite mass
$M$ the piston's motion is approximately oscillatory for an $O(√ M)$
time period. Shah, et. al.~\cite{10.1073/pnas.1706341114} explain that
in slow-fast systems the Gibbs volume entropy of the fast subsystem is
conserved, so ergodicity of the fast subsystem frustrates ergodicity
of the whole. Indeed, for the Nos{é}-Hoover thermostat, it is proven
that in the decoupled limit of $ε=0$, the thermostat's state
oscillates in a potential well $U$ where $U$ is an analogue of the
free energy of the fast subsystem~(see eq. \ref{eq:nh-g0}).

The outline of this paper is: \S~\ref{sec:st} introduces a definition
of a single thermostat for a hamiltonian system; \S~\ref{sec:prelims}
establishes notation and terminology to discuss $1$ degree-of-freedom
hamiltonian systems; \S~\ref{sec:wcst} derives a Poincar{é} return
map, invariant symplectic form, \& hamiltonian for the averaged
thermostated vector field; \S~\ref{sec:ex} proves
Theorem~\ref{thm:intro-thm-1} using the results of the previous
section.

\section{Single Thermostats}
\label{sec:st}

Here and throughout the paper, an object is \defn{smooth} if it is
$C^r$ for some $r>7$ and a map is \defn{proper} if the pre-image of
every compact set is compact.

Let $Σ$ be a smooth manifold, $\cotangent Σ$ its cotangent bundle and
$\pb{}{}$ the canonical Poisson bracket. Let $π : P → \cotangent Σ$ be
a trivial line (or circle) bundle over $\cotangent Σ$. One should
think of $P$ as an extended phase space which models the state of a
mechanical system with points $(q,p) ∈ \cotangent Σ$ and the
thermostat's local state with $ξ ∈ π^{-1}(q,p)$. Triviality of $P$
implies that there is a unique pullback of $\pb{}{}$ to $P$ such that
$ξ$ is a locally-defined Casimir; this pullback is denoted by
$\pb{}{}$, too. This pullback is characterized by the property that if
$(q_i,p_i,ξ)$ are local coordinates on $P$ such that $(q_i,p_i)$ are
canonical coordinates on $\cotangent Σ$, then
$\pb{p_i}{q_j} = δ_{ij} = -\pb{q_j}{p_i}$ and all other brackets
vanish.

Given a smooth hamiltonian $H : \cotangent Σ → \R$, let
$X_H = \pb{H ⎄ π}{⋅}$ be the hamiltonian vector field lifted to
$P$. Say that a probability measure
\begin{equation}
  \label{eq:dmu}
  \D{μ}_{β} = Z_1(β)^{-1} \exp(-β G_{β}(q,p,ξ) )\, \D q\, \D p\, \D ξ
\end{equation}
projects to the probability measure
\begin{equation}
  \label{eq:dm}
  \D{m}_{β} = Z(β)^{-1} \exp(-β H(q,p) )\, \D q\, \D p
\end{equation}
if $\D{m}_{β} = \int_{ξ} \D{μ}_{β}$, i.e. if $\D{m}_{β}$ is a marginal
of $\D{μ}_{β}$. It is a natural convention in the literature on
thermostats to postulate that the invariant measure $\D{μ}_{β}$ of
the thermostated vector field projects to the Gibbs-Boltzmann
probability measure $\D{m}_{β}$ of the mechanical system. Somewhat
surprisingly, the main result of this paper does not require such an
assumption.

\begin{definition}
  \label{def:thermostat}
  A smooth, $(ε,β)$-dependent, vector field $Τ=Τ_{ε,β}=Τ_0 + O(ε)$ on
  $P$ is a \defn{thermostat} for $H$ if there is a smooth probability
  measure $\D{μ}_{β}$ on $P$ such that the following holds
  \begin{enumerate}
  \item\label{def:thermostat-invariance} $\D{μ}_{β}$ is invariant for $Y_{ε} = X_H + ε Τ$ for all $ε$;
  \item\label{def:thermostat-proper} $G_{β}$ is proper for all $β>0$;
  \item\label{def:thermostat-thermostat} there exists an interval of regular values of $H$,
    $[c_-,c_+]$, constants $d_- < d_+$ and a connected component
    $P_{c,d} ⊂ (H × ξ)^{-1}([c_-,c_+] × [d_-,d_+])$ such that
    \begin{enumerate}
    \item the average value of $\kp{\D ξ}{Τ_0}$ is of opposite sign on
      $P_{c,d} ∩ H^{-1}(c_-)$ and $P_{c,d} ∩ H^{-1}(c_+)$;
    \item the average value of $\kp{\D H}{Τ_0}$ is of opposite sign on
      $P_{c,d} ∩ ξ^{-1}(d_-)$ and $P_{c,d} ∩ ξ^{-1}(d_+)$.
    \end{enumerate}
  \end{enumerate}
\end{definition}

Part~\itref{def:thermostat-thermostat} may be roughly translated into
the following description: In the $H-ξ$ plane there is a box
$[c_-,c_+] × [d_-,d_+]$ and (a) on the left (resp. right) edge $ξ$
decreases (resp. increases) on average; (b) on the bottom (resp. top)
edge $H$ increases (resp. decreases) on average. The condition (a)
says that, on average, the thermostat state $ξ$ does not increase or
decrease indefinitely; (b) says that the hamiltonian $H$ is heated at
low energy and cooled at high energy. The thermostat vector field may
be temperature-dependent; in principle, one should insist that the
average temperature along almost every orbit of the thermostated
vector field $Y_{ε}$ be $T=1/β$ (such is the case for the four
thermostats in Theorem~\ref{thm:intro-thm-1}), but even with such a
hypothesis condition~\ref{def:thermostat-thermostat} seems
necessary. The thermostat vector field may also be $ε$-dependent; this
is a natural assumption because it implies that an averaging
transformation transforms a thermostat vector field to a thermostat
vector field (see definition \ref{def:first-order-averaging-transfm}
and the subsequent paragraph below). In the following, the
$β$-dependence of the thermostat vector field $Τ$ is suppressed: this
is partly for expedience, partly because $β>0$ is fixed throughout,
and partly because the notation is meant to suggest that the
thermostat vector field $Τ$ is forcing the hamiltonian as if it were
in contact with a heat bath at temperature $T$.

The requirement that the (log of the) density of the invariant measure
$\D{μ}_{β}$ be proper, \itref{def:thermostat-proper}, may appear
unnatural at first sight. However, it implies that the hamiltonian $H$
is proper--which is natural--and it is a property shared by all
examples in the literature known to the author (see section
\ref{sec:ex} below).

Definition~\ref{def:thermostat} encompasses that of
Ramshaw~\cite[Section V]{PhysRevE.92.052138}. This latter work, which
develops a formalism that encapsulates earlier single thermostats,
posits specific forms of $G$ and $Τ$.

\section{Preliminary Materials}
\label{sec:prelims}

This section establishes notation and terminology for subsequent
sections. Throughout, $Σ = \R$ or $\T^1$ and
$\cotangent Σ = \set{(q,p) \mid q ∈ Σ, p ∈ \R}$ is the cotangent
bundle. A smooth function $H : \cotangent Σ → \R$ (also called a
hamiltonian function) is \defn{Morse} if, at each critical point its
Hessian is non-degenerate; it is a \defn{topological Morse} function
if each critical point has a neighbourhood homeomorphic to a
neighbourhood of a critical point of a Morse function $H'$ and $H'$ is
conjugate to $H$ by this homeomorphism; $H$ is \defn{proper} if the
pre-image of each compact set is compact; $H$ is \defn{mechanical} if
$H(q,p) = ½ p ² + V(q)$ and \defn{quasi-mechanical} if
$H(q,p) = F(p) + V(q)$ where $F$ is even, $F(0)=0$ and $F'(p)/p > 0$
for all $p ≠ 0$. The following definition summarizes the key
properties needed in the present paper.

\begin{definition}[Well behaved]
  \label{def:h-is-well-behaved}
  If $H : \cotangent Σ → \R$ is a smooth hamiltonian that satisfies
  \begin{enumerate}
  \item\label{it:h-is-well-behaved-proper} $H$ is proper;
  \item\label{it:h-is-well-behaved-Morse} $H$ is (topologically) Morse;
  \item\label{it:h-is-well-behaved-finite} $H$ has a compact critical set;
  \item\label{it:h-is-well-behaved-no-maxima} $H$ has no local maxima;
  \item\label{it:h-is-well-behaved-unbounded-temp} for all odd
    integers $k ≥ 1$ and $β>0$, $\gibbsexp{β}{p^{k-1}}$ and
    $\gibbsexp{β}{p^k · H_p}$ exist and
    \begin{equation}
      \label{eq:h-is-well-behaved-unbounded-temp}
      \limsup_{β ↘ 0} \dfrac{\gibbsexp{β}{p^k · H_p}}{\gibbsexp{β}{p^{k-1}}} = +∞,
    \end{equation}
    where $\gibbsexp{β}{φ} = \iint_{\cotangent Σ} φ\, \D{m}_{β}$ is
    the expected value of $φ$ with respect to the Gibbs-Boltzmann measure~\eqref{eq:dm};
  \end{enumerate}
  then $H$ is said to be \defn{(topologically) well-behaved}.
\end{definition}

Conditions (1-3) are fairly natural assumptions. Condition (4) is
modelled on the same property for mechanical hamiltonians. Condition
(5) is also, given the setting: if $H$ is mechanical, the stated ratio
of mean values is $T$ (independent of $k$), and the assumption is that
the ``effective'' temperature
$\gibbsexp{β}{p^k · H_p}/\gibbsexp{β}{p^{k-1}}$ of the system should
converge to infinity as $T$ does provided that $H$ is
``well-behaved''.

This paper is primarily concerned with smooth, topologically
well-behaved hamiltonians $H$. $H$ has only local minima and saddle
critical points by
condition~\ref{it:h-is-well-behaved-no-maxima}. Moreover, properness
implies that $H$ is bounded below and it attains that lower bound
(which will be assumed to be $0$ henceforth). Conditions
(\ref{it:h-is-well-behaved-Morse}+\ref{it:h-is-well-behaved-finite})
imply that $H$ has only finitely many critical points. If $c$ is a
local minimum, then $H^{-1}(c)$ is a union of a finite number $r_c$ of
regular circles and $k_c$ minimum points. A neighbourhood $N_c$ of the
critical level is a disjoint union of $r_c$ annuli and $k_c$ disks. On
the other hand, if $c$ is a saddle critical value, then since $H$ has
no local maxima, $H^{-1}(c)$ also has a simple description: there are
$r_c ≥ 0$ regular circles, and $s_c > 0$ singular path
components. When $Σ=\R$, the $i$-th~%
singular component of $H^{-1}(c)$ consists of $k_{c,i}+1 > 1$ circles
pinched at $k_{c,i}$ distinct points. A small neighbourhood $N_c$ of
$H^{-1}(c)$ is a disjoint union of $r_c$ annuli and $s_c$ disks where
the $i$-th disk has $k_{c,i}+1$ disjoint, smaller disks removed from
its interior. The boundary of $N_c$ consists of the boundary of those
deleted smaller disks and the ``lower half'' of the annuli boundaries
(which make up $H^{-1}(c-ε)$) and the boundary of the larger disk and
``upper half'' of the annuli boundaries (which make up
$H^{-1}(c+ε)$). When $Σ=\T^1$, the above description holds except for
the largest saddle critical value: in that case, a neighbourhood $N_c$
of $H^{-1}(c)$ is easiest to describe: it is a cylinder with $k_c$
disjoint disks removed. The boundary of $N_c$ consists of $2$
essential circles ($=H^{-1}(c+ε)$) and $k_c$ inessential circles
($=H^{-1}(c-ε)$); $N_c$ retracts onto $H^{-1}(c)$, which is
$\max\set{2,k_c}$ circles pinched at $k_c$ points. Finally, if $c$ is
a critical value of mixed type (i.e. $H^{-1}(c)$ contains both a local
minimum and a saddle), then $H^{-1}(c)$ contains $r_c$ regular
circles, $k_c$ local minima and $s_c$ saddle components and the above
descriptions of a neighbourhood $N_c$ are combined. Because the saddle
components are most important for the purposes here, a critical value
$c$ will be said to be a saddle value if $s_c > 0$, i.e. if
$H^{-1}(c)$ contains a saddle.

The preceding paragraph implies that the coarse topological structure
of the level-sets of $H$ can be summarized in a directed tree $Γ_H$
with the following structure: (see figure~\ref{fig:gamma-h})
\begin{enumerate}
\item[$Σ=\R$:] $Γ_H$ is a finite tree with each branch either
  terminating at a vertex (a local minimum) or branching into $s_c>1$
  separate branches (a saddle), the root vertex is labeled $∞$ and the
  highest vertices are labeled $0$; or
\item[$Σ=\T^1$:] $Γ_H$ is obtained from a finite tree similar to that
  described in the first case by splitting the root branch and vertex
  in two (and labeling the latter as $± ∞$).
\end{enumerate}
Each edge of $Γ_H$ is naturally homeomorphic to a closed interval by
$H$, and $H$ partially orders the graph, too.

The graph $Γ_H$ has a second, equally valuable description. Each point
$γ ∈ Γ_H$ is a path-connected component of a level set of $H$. When
$\D{H} |_{γ}$ does not vanish (i.e. when $γ$ lies in the interior of
an edge), $γ$ is a circle and an orbit of the hamiltonian flow $φ^t$
of $H$. Moreover, there is a canonical quotient map $ψ$ and functions
$I,\tilde{I}$ such that
\[
  \xymatrix @R=2pc @C=4pc @M=5pt {
    \coprod \T^1 × \R \ar[drr] \\
    \ar[u]_{(θ,I)} L \ar@{^{(}->}[r]^{}\ar[d]^{ψ|L}                                       & \cotangent Σ \ar[r]_{I}\ar[d]_{ψ} & \R \\
                             B_H \ar@{^{(}->}[r]^{}                                       & Γ_H \ar[ur]_{\tilde{I}}
  }
\]
commutes. The quotient map $ψ$ is the quotient map obtained from the
equivalence relation $\sim$ where $X \sim X'$ iff $H(X)=H(X')$ and $X$
and $X'$ lie in the same path-connected component of
$H^{-1}(H(X'))$. The function $\tilde{I}$ is defined by
\[ 2 π \tilde{I}(γ) = \oint_{γ} p\, \D{q}. \] $\tilde{I}$ is
continuous on $Γ_H$ less the set of saddle vertices. At a saddle
vertex $σ$ one has the identity
\[ \lim_{γ ↘ σ} \tilde{I}(γ) = \sum_{γ} \lim_{γ ↗ σ} \tilde{I}(γ), \]
where the right-hand sum is the sum over all edges incoming to
$σ$. This also holds for vertices of local minima, with the convention
that the sum over an empty set is $0$ (i.e. $\tilde{I}$ is continuous
at local minimum vertices).

If $B_H = Γ_H - V_H$ is the set of points that are not vertices,
i.e. $B_H$ is the union of the interiors of the edges of $Γ_H$, and
$L = ψ^{-1}(B_H)$, then $ψ | L$ is a proper submersion whose fibres
are circles. Classical constructions yield the existence of
\defn{angle-action} variables $(θ,I) : L → \coprod \T^1 × \R$ where
the disjoint union is taken over the edge set of $Γ_H$~\cite[\S
50]{MR997295}. In these variables, $H=H(I)$ and $H_I>0$ since
$\tilde{I}$ is monotone increasing in $H$. If $σ$ is a saddle vertex,
then as $γ → σ$ (from above or below), $H_I(\tilde{I}(γ)) → 0$ since
the period goes to $∞$; if $σ$ is a local minimum vertex, then as
$γ ↘ σ$, $H_I(\tilde{I}(γ)) → ω_{σ} > 0$ where $ω_{σ}$ is the
frequency of the linearized oscillations at $σ$. It follows that the
function
\begin{equation}
  \label{eq:kappa}
  Κ(I) = I · H_I(I)
\end{equation}
is a continuous, non-negative function that vanishes only on the
vertex set of $Γ_H$ and is smooth on
$B_H$. Condition~\ref{it:h-is-well-behaved-unbounded-temp} of the
well-behaved definition with $k=1$ implies that $Κ$ maps a root edge
of $Γ_H$ surjectively onto $[0,∞)$ (see
Lemma~\ref{lem:nh-is-a-separable-thermostat}).

Let it be noted that $Γ_H$ is defined for all $C^2$ proper
hamiltonians, but it is not as nice in the general case. If $H$ is
topologically well-behaved, then $Γ_H$ has the structure and
properties described above.

\begin{figure}[h]
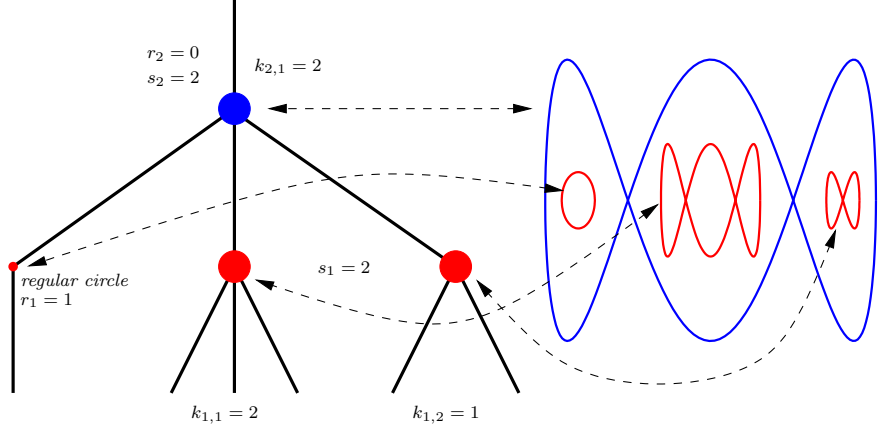

  \label{fig:gamma-h}
  \caption{An example of $Γ_H$ and the contours of $H$ with two
    critical levels at heights $c=1$ (red) \& $2$ (blue). The regular
    circle on the left branch is a marked point, not a vertex.}
  \ltxfigure{gamma-h-illustrated.tex}{0.95\textwidth}{!}
\end{figure}

Background references for this material
include~\cite{scopus2-s2.0-30244486810} and \cite[chapter
2]{MR2036760}. A construction of the Reeb graph $Γ_H$ for compact
surfaces is described there.

\section{Weakly Coupled Single Thermostats}
\label{sec:wcst}

In the sequel, $H : \cotangent Σ → \R$ is a proper, smooth function;
$Τ$ is a thermostat for $H$ in the sense of definition
\ref{def:thermostat} and $\D{μ}_{β}$ is an invariant probability
measure in the same sense.

In the variables $(θ,I,ξ)$ on $π^{-1}(L) ⊂ P$, one has
\begin{align}
  \label{eq:aa-nf}
  X_H &= H_I\, ∂_{θ}, &&& Τ &= a\, ∂_{θ} + b\, ∂_{I} + c\, ∂_{ξ} \\\notag
  &&&&&= a_0\, ∂_{θ} + b_0\, ∂_{I} + c_0\, ∂_{ξ} + O(ε)
\end{align}
where $a,b,c$ are smooth functions of $(θ,I,ξ;ε)$ and $a_0,b_0,c_0$
are independent of $ε$. The invariance of $\D{μ}_{β}$ implies that
\begin{align}
  \label{eq:invariance1}
  \kp{\D G}{X_H} &≡ 0, &&& \textrm{so } &&&& G &= G(I,ξ) &&& \textrm{and}\\
  \label{eq:invariance2}
  β \kp{\D G}{Τ} - \divergence{Τ} &≡ 0, &&& \textrm{so } &&&& a_{θ} &= β\, b\, G_I - b_I + β\, c\, G_{ξ} - c_{ξ}.
\end{align}

\subsection{Averaging}
\label{sec:averaging}

Let $x̄$ denote the mean value of $x$ over $θ$:
$x̄(I,ξ) = \frac{1}{2 π} \int_0^{2 π} x(θ,I,ξ) \D{θ}$. If
$x = x̄$, then the over-bar will be omitted. Equations
\ref{eq:invariance1}--\ref{eq:invariance2} imply that
\begin{align}
  \label{eq:invariant3}
  0 &≡ β\, \bar{b} \, G_I - \bar{b}_I + β\, c̄\, G_{ξ} - c̄_{ξ}.
\end{align}

Let $P_{c,d} ⊂ H^{-1}([c_-,c_+]) \cap ξ^{-1}([d_-,d_+])$ be the
compact, connected component from
condition~\itref{def:thermostat-thermostat} of
definition~\ref{def:thermostat}. The hamiltonian $H$ is critical-point
free on $P_{c,d}$, so $H_I ≠ 0$ on this set. Moreover, it is clear
that the condition~\itref{def:thermostat-thermostat} is satisfied for
all sufficiently small perturbations of the solid torus $P_{c,d}$,
too.

\begin{definition}
  \label{def:first-order-averaging-transfm}
  A \defn{first-order averaging transformation} is a near-identity map
  $F_{ε} : P_{c,d} → π^{-1}(L)$ such that
  \begin{enumerate}
  \item $F_{ε} = 1 + O(ε)$;
  \item $F_{ε\, *}\, \left( X_H + ε Τ \right) = X_H + ε \bar{Τ}_0 + O(ε^2)$; and
  \item $F_{ε}^* \D{μ}_{β} =  \D{μ}_{β}$,
  \end{enumerate}
  for all $ε, β > 0$.
\end{definition}

Conditions (2+3) imply that the transformed thermostated vector field
$\tilde{Y}_{ε} := F_{ε\, *} Y_{ε}$ is itself a thermostated vector
field and $\tilde{Τ}=\bar{Τ}_0 + O(ε)$ is also a thermostat for $H$
according to definition~\ref{def:thermostat}. For reference, the
averaged vector field is $\bar{Y}_{ε} = X_H + ε \bar{Τ}$ and it
coincides with $\tilde{Y}_{ε}$ up to $O(ε^2)$.

\begin{lemma}
  \label{lem:first-order-averaging-transfm}
  There exists $c', d'$ and a first-order averaging transformation $F_{ε} : P_{c',d'} → π^{-1}(L)$ for all $ε$ sufficiently small.
\end{lemma}
\begin{proof}
  The existence of a near-identity transformation $f_{ε}$ satisfying
  conditions (1+2) of
  definition~\ref{def:first-order-averaging-transfm} follows from
  averaging theory~\cite{MR2316999}. Indeed, one verifies that
  \begin{align}
    \label{eq:first-order-averaging-transfm0}
    f_{ε}(θ,I,ξ) & = (θ,I,ξ) + ε\, (φ,ι,η),
                 &  &  & ι & = \dfrac{1}{2 π H_I} × \int_0^{θ} \left( \bar{b}(I,ξ) - b(I,t,ξ) \right)\, \D{t}                \\\notag
                 &  &  &   & φ & = \dfrac{1}{2 π H_I} × \int_0^{θ} \left( \bar{a}(I,ξ) - a(I,t,ξ) - H_{II} ι \right)\, \D{t} \\\notag
                 &  &  &   & η & = \dfrac{1}{2 π H_I} × \int_0^{θ} \left( \bar{c}(I,ξ) - c(I,t,ξ) \right)\, \D{t} 
  \end{align}
  is such a transformation provided that $P_{c,d}$ is shrunk by a fixed amount, i.e. $c'_{±} = c_{±} + O(ε_0)$ and similarly for $d'_{±}$.
  It is straightforward to verify that $f_{ε}^*\, \D{μ}_{β} = (1+O(ε^2))\, \D{μ}_{β}$--this follows from the fact that the averaged vector field $\bar{Y}_{ε}$ preserves $\D{μ}_{β}$ and $\tilde{Y}_{ε} = \bar{Y}_{ε} + O(ε^2)$.

  Moser's isotopy trick~\cite[Lemma 2]{scopus2-s2.0-84968494100}
  implies there is a diffeomorphism $g_{ε} = 1 + O(ε^2)$ such that
  $g_{ε}^* f_{ε}^*\, \D{μ}_{β} = \D{μ}_{β}$. The statement of that
  lemma assumes the domain is a unit cube $[0,1]^n ⊂ \R^n$, but the
  proof that constructs the diffeomorphism works without alteration
  for the case where one or more coordinates are $1$-periodic,
  too. The lemma also requires that the volume forms
  $f_{ε}^*\, \D{μ}_{β}$ and $\D{μ}_{β}$ coincide on a neighbourhood of
  the boundary of the solid torus $P_{c,d}$; by means of a $C^{∞}$
  bump function that is unity on $P_{c',d'} ⊂ P_{c,d}$ and vanishes on
  a neighbourhood of the boundary of $P_{c,d}$ this can be
  achieved. It follows that $F_{ε} = f_{ε} \circ g_{ε}$ satisfies
  conditions (1--3) of
  definition~\ref{def:first-order-averaging-transfm}.
\end{proof}

\begin{remark}
  \label{rem:differentiability}
  If $H$ is $C^r$ in the natural mechanical coordinates $(q,p)$ on
  $\cotangent Σ$, then $I, ξ$ are $C^r$ functions but $θ$ is only
  $C^{r-1}$ (since it is a normalized time along the flow of the
  $C^{r-1}$ vector field $X_H$). This implies that the averaging
  transformation $f_{ε}$ is $C^{r-2}$; Moser's construction does not
  reduce differentiability~\cite[p. 291]{scopus2-s2.0-84968494100} so
  $g_{ε}$ and hence the first-order averaging transformation $F_{ε}$
  is $C^{r-2}$ and hence $\tilde{Y}_{ε}$ is $C^{r-3}$ while
  $\bar{Y}_{ε}$ is $C^{r-1}$.

  In the following, it is assumed that the constants $c_{±}, d_{±}$ in
  definition~\ref{def:thermostat} are chosen so that there is a
  first-order averaging transformation $F_{ε}$ defined on $P_{c,d}$
  for all $ε$ sufficiently small.
\end{remark}

Define
\begin{align}
  \label{eq:volume-time-rescaling}
  \tilde{λ}_{ε} &= \kp{\D{θ}}{\tilde{Y}_{ε}}, &&& \tilde{Z}_{ε} &= \tilde{λ}_{ε}^{-1}\, \tilde{Y}_{ε}, &&& \D{\tilde{Λ}}_{ε,β} &= \tilde{λ}_{ε}\, \D{μ}_{β};\\\notag
  \bar  {λ}_{ε} &= \kp{\D{θ}}{\bar  {Y}_{ε}}, &&& \bar  {Z}_{ε} &= \bar  {λ}_{ε}^{-1}\, \bar  {Y}_{ε}, &&& \D{\bar  {Λ}}_{ε,β} &= \bar  {λ}_{ε}\, \D{μ}_{β}
\end{align}
The scalars $\tilde{λ}_{ε}$ and $\bar{λ}_{ε}$ both equal $H_I + ε \bar{a} + O(ε^2)$ so their difference is $O(ε^2)$. The vector field $\tilde{Z}_{ε}$ (resp. $\bar{Z}_{ε}$) preserves the volume form $\D{\tilde{Λ}}_{ε,β}$ (resp. $\D{\bar  {Λ}}_{ε,β}$) and these volume forms differ by $O(ε^2)$. It is assumed that $ε$ is sufficiently small that neither scalar vanishes on $P_{c,d}$.

Let
\begin{equation}
  \label{eq:poincare-section}
  S = S_{c,d} = P_{c,d} \cap \set{ θ ≡ 0 \bmod 2 π },
\end{equation}
which is a Poincar{é} section for $\tilde{Z}_{ε}$ and
$\bar{Z}_{ε}$. Let $\tilde{ρ}_{ε}$ (resp. $\bar {ρ}_{ε}$) be the
return map of $\tilde{Z}_{ε}$ (resp $\bar {Z}_{ε}$). The return map of
each vector field is the flow map restricted to $S$ evaluated at time
$2 π$, due to the normalization in
\eqref{eq:volume-time-rescaling}. The return map $\tilde{ρ}_{ε}$
(resp. $\bar {ρ}_{ε}$) preserves the area form
$\D{\tilde{m}}_{ε,β} = \left. \tilde{Λ}_{ε,β} \right|_S$
(resp. $\D{\bar {m}}_{ε,β} = \left. \bar {Λ}_{ε,β} \right|_S$)
(c.f.~\cite[\S 3]{MR2299758}). It follows
from~\eqref{eq:volume-time-rescaling} that
$\D{\tilde{m}}_{ε,β} = \D{\bar {m}}_{ε,β} + O(ε^2)$ and that
$\D{\bar {m}}_{ε,β} = Z(β)^{-1}\, H_I\, \exp(-β G)\, \D{I} ∧ \D{ξ} +
O(ε)$ (note the difference in orders).

\begin{lemma}
  \label{lem:symplectic-poincare-maps}
  There exist diffeomorphisms $\tilde{g}_{ε} = 1 + ε \tilde{g}_1 +
  O(ε^2)$, $\bar  {g}_{ε} = 1 + ε \bar  {g}_1 + O(ε^2)$ of $S$ such
  that for all $ε$ sufficiently small
  \begin{enumerate}
  \item\label{it:symplectic-poincare-maps-1} $\tilde{g}_{ε} - \bar  {g}_{ε} = O(ε^2)$, i.e. $\tilde{g}_1 = \bar{g}_1$;
  \item\label{it:symplectic-poincare-maps-2} $\tilde{g}_{ε}^*\, \D{\tilde{m}}_{ε,β} = ω_{β} = \bar
    {g}_{ε}^*\, \D{\bar  {m}}_{ε,β}$ where $ω_{β} = Z(β)^{-1}\, H_I\, \exp(-β G)\, \D{I} ∧ \D{ξ}$;
  \item\label{it:symplectic-poincare-maps-3} $\tilde{ψ}_{ε} = \tilde{g}_{ε}^{-1} \circ \tilde{ρ}_{ε} \circ \tilde{g}_{ε}$ and $\bar  {ψ}_{ε} = \bar  {g}_{ε}^{-1} \circ \bar  {ρ}_{ε} \circ \bar  {g}_{ε}$ preserve the symplectic form $ω_{β}$;
  \item\label{it:symplectic-poincare-maps-4} $\tilde{ψ}_{ε} - \bar  {ψ}_{ε} = O(ε^2)$.
  \end{enumerate}
\end{lemma}
\begin{proof}
  The existence of $\tilde{g}_1$ (resp. $\bar{g}_1$) that satisfies
  condition \itref{it:symplectic-poincare-maps-2} is a consequence of
  Moser's isotopy trick~\cite{scopus2-s2.0-84968494100} and the
  remarks preceding this lemma. \itref{it:symplectic-poincare-maps-1}
  and \itref{it:symplectic-poincare-maps-4} likewise follow from the
  same remarks. \itref{it:symplectic-poincare-maps-3} is a simple
  consequence of \itref{it:symplectic-poincare-maps-2} and the same
  remarks.
\end{proof}

\begin{remark}
  \label{rem:symplectic-isotopy}
The family of symplectic maps, $ε \to \tilde{ψ}_{ε}$ (resp.
$ε \to \bar {ψ}_{ε}$), is a symplectic isotopy to the
identity. Therefore, $\tilde{ψ}_{ε}$ (resp. $\bar {ψ}_{ε}$) is the
time-$2 π$ map of a time-dependent hamiltonian vector field
$\tilde{R}_{ε}$ (resp. $\bar {R}_{ε}$) on $(S, ω_{β})$. Indeed, the
vector field may be constructed from a hamiltonian suspension flow:
define an equivalence relation $\sim$ on $M=\R × \R × S$ by
$(k,t+2 π,s) \sim (k,t,\tilde{ψ}_{ε}(s))$ for all
$(k,t,s) ∈ \R × \R × S$. The symplectic form
$Ω_{β} = -\D{k} ∧ \D{t} + ω_{β}$ is invariant, as is the hamiltonian
vector field $\tilde{X} = ∂_t$ with hamiltonian $\tilde{H}=k$. The
symplectic manifold $(M/\hspace{-.5em}\sim, Ω_{β})$ is symplectomorphic to
$\R × \R/2 π \Z × S$ via a ``near identity'' symplectomorphism. The
pullback of $\tilde{X}$ (resp. $\tilde{H}$) is $∂_t + \tilde{R}_{ε}$
(resp. $\tilde{H}_{ε,β} = k + \tilde{G}_{ε,β}$) where $\tilde{R}_{ε}$
annihilates both $\D k$ and $\D t$. It is clear that
\begin{align}
  \label{eq:pse}
  \tilde{R}_{ε} &= \dfrac{ε}{H_I} \left( \bar{b}_0\, ∂_I + \bar{c}_0\, ∂_{ξ} \right) + O(ε^2) \\\notag
  &=: ε \bar  {R}_0 + O(ε^2),
\end{align}
and that $\bar {R}_{ε} = ε \bar {R}_0 + O(ε^2)$, too. Of course, the
second-order remainder terms differ and may depend on $(k,t)$, too. In
the sequel, one normalizes the family of hamiltonian functions by
requiring them to vanish identically at some fixed point in $S$.
\end{remark}

\begin{lemma}
  \label{lem:ave-transformation}
  The hamiltonian function $\tilde{G}_{ε,β} = ε \tilde{G}_{0,β} + O(ε^2)$ (resp. $\bar  {G}_{ε,β} = ε \bar  {G}_{0,β} + O(ε^2)$) of $\tilde{R}_{ε}$ (resp. $\bar  {R}_{ε}$) satisfies
  \begin{enumerate}
  \item\label{it:ave-transformation-1} $\tilde{G}_{0,β} = \bar  {G}_{0,β}$; and
  \item\label{it:ave-transformation-2}
    $
      \bar  {G}_{0,β}(I,ξ) = \displaystyle\int\limits_{d_-}^{ξ} Z(β)^{-1}\, \exp(-β G(I,x))\, \bar{b}_0(I,x)\, \D{x} \bmod I.
    $
  \end{enumerate}
\end{lemma}

\begin{proof}
  Condition~\itref{it:ave-transformation-1} follows from the equality of $\tilde{R}_{ε}$ and $\bar  {R}_{ε}$ to second-order in $ε$.
  The formula for $\bar{G}_{0,β}$ follows from computing the interior product of $\bar  {R}_{ε}$ with the symplectic form $ω_{β}$.
\end{proof}

\subsection{Critical points of the averaged hamiltonian}
\label{sec:crit-pts-ave-ham}

In lemmas~\ref{lem:ave-crit-pt} \& \ref{lem:ave-crit-pt-index}, one
uses that the Poincar{é} section $S_{c,d}$ is connected. This follows
from the connectedness of $P_{c,d}$ in
definition~\ref{def:thermostat}.

\begin{lemma}
  \label{lem:ave-crit-pt}
  $\bar{G}_{0,β}$ has a critical point $(I_0,ξ_0)$ in the interior of $S_{c,d}$.
\end{lemma}

\begin{proof}
  Let $\bmean{b}(H,ξ) = \bar{b}_0(I(H),ξ)$ and similarly for
  $\bmean{c}$. The averaged value of $\kp{\D{H}}{Τ}$
  (resp. $\kp{\D{ξ}}{Τ}$) is $H_I \bar{b}_0 + O(ε)$ (resp.
  $\bar{c}_0 + O(ε)$). We will assume, without loss of generality,
  that
  \begin{enumerate}
  \item $s \bmean{b}(u,d_s) > 0$ for all $s ∈ \set{+,-}$, $u ∈
    [c_-,c_+]$; and
  \item $s \bmean{c}(c_s,v) > 0$ for all $s ∈ \set{+,-}$, $v ∈
    [d_-,d_+]$.
  \end{enumerate}
  Define the map $r=r(H,ξ)$, by
  \begin{align}
    \label{eq:r}
    r(H,ξ) &= (\bmean{c}(H,ξ), \bmean{b}(H,ξ)), &&& r : W → \R ²,
  \end{align}
  $W = [c_-,c_+] × [d_-,d_+]$.
  By definition \ref{def:thermostat} and the fact that $H_I$ does not
  change sign, for all $ε$ sufficiently small,
  $r(∂ W) ⊂ \R ² - \set{0}$. Since $W$ is affinely homeomorphic to the
  square $[-1,1] × [-1,1]$, the present lemma now follows from the
  topological lemma \ref{lem:top-crit-pt}.
\end{proof}

\def\newGamma{{\mathrm{C}}}
\let\oldGamma=\Gamma
\let\Gamma=\newGamma
\begin{lemma}
  \label{lem:ave-crit-pt-index}
  Assume that the critical points of $\bar{G}_{0,β}$ are isolated and
  that for some $Β,Γ ∈ \set{-1,1}$,
    \begin{enumerate}
    \item $s\, Β\, \bmean{b}(u,d_s) > 0$ for all $s ∈ \set{+,-}$, $u ∈ [c_-,c_+]$; and
    \item $s\, Γ\, \bmean{c}(c_s,v) > 0$ for all $s ∈ \set{+,-}$, $v ∈ [d_-,d_+]$.
  \end{enumerate}
  Then the sum of the indices of the critical points of $\bar{R}_0$ in
  $S_{c,d}$ is $-Β Γ$. In particular, if the
  critical points are non-degenerate and $Β Γ=-1$, then there is an
  elliptic critical point.
\end{lemma}
\begin{proof}
  From lemma \ref{lem:ave-crit-pt}, there is at least one critical
  point in the pre-image of $W$. The sum of the indices of the
  critical points is $-Β Γ × \degree{r|∂ W} = -Β Γ$.
\end{proof}

\begin{lemma}
  \label{lem:top-crit-pt}
  Let $W = [-1,1] × [-1,1] ⊂ \R ²$ and $f : W → \R ²$ be a continuous
  map such that
  \begin{align}
    \notag
    f(u,v) &= (x(u,v),y(u,v)) &\text{ and } s y(u,s) > 0,\ s x(s,v) > 0 \\\label{eq:top-crit-pt}
           &                  &\text{ where } s ∈ \set{1,-1},\ u,v ∈ [-1,1].
  \end{align}
  Then, $0 ∈ f(W)$.
\end{lemma}
\begin{proof}
  This is clear; see figure~\ref{fig:crit-pt}.
  \begin{figure}
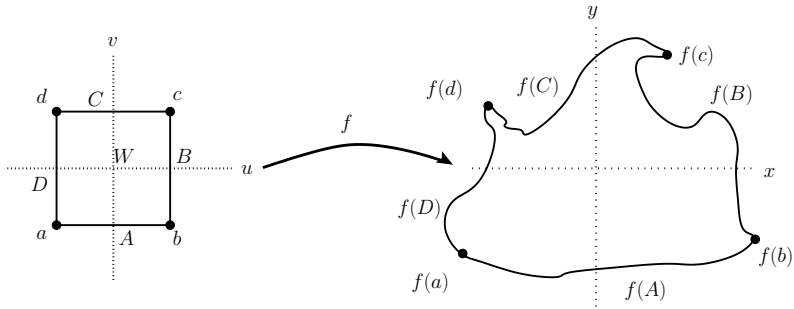

    \label{fig:crit-pt}
    \caption{The domain and image of $f$ in Lemma~\ref{lem:top-crit-pt}.}
    \ltxfigure{crit-pt.tex}{0.8\textwidth}{!}
  \end{figure}
\end{proof}

\begin{theorem}
  \label{thm:main-thm}
  Assume that $H$, $G$ and $Τ$ are $C^r$ for some $r>7$ and either
  \begin{enumerate}
  \item\label{cond:thm-main-thm-1} that conditions (1--2) of lemma
    \ref{lem:ave-crit-pt-index} hold with $Β Γ =-1$ and the period of
    the vector field $\bar{R}_0$ is not constant in a neighbourhood of
    the elliptic critical point $(I_0,ξ_0)$; or
  \item\label{cond:thm-main-thm-2} the averaged hamiltonian
    $\bar{G}_{0,β}$ is proper and $\bar{R}_0$ has a non-constant
    period,
  \end{enumerate}
  then for all $ε$ sufficiently small the Poincar{é} return map
  $\tilde{ψ}_{ε} : S → S$ has a positive measure set of invariant
  circles.
\end{theorem}

\begin{corollary}
  \label{cor:non-ergodcity}
  Assume the hypotheses of Theorem~\ref{thm:main-thm}.
  For all $ε$ sufficiently small, the thermostated vector field
  $Y_{ε}$ (\ref{def:thermostat}) possesses a positive measure set of
  invariant tori.
\end{corollary}

\begin{proof}[Proof of Theorem \ref{thm:main-thm} \& Corollary~\ref{cor:non-ergodcity}]
  Let $(I_0,ξ_0)$ be either (1) an elliptic critical point; or (2) a
  regular point. In either case, there exists action-angle variables
  $(J,ρ)$ for $\tilde{G}_{0,β}$ which are defined in a neighbourhood
  of $(I_0,ξ_0)$ (with an unimportant singularity at $(I_0,ξ_0)$ in
  case~\ref{cond:thm-main-thm-1}). In these coordinates,
  $\tilde{G}_{0,β} = \tilde{G}_{0,β}(J)$. Since the period of
  $\tilde{R}_0$ is
  $2 π \left( \didi[{\tilde{G}_{0,β}}]{J} \right)^{-1}$, non-constancy
  of the period implies that there is some neighbourhood of
  $(I_0,ξ_0)$ on which $\frac{∂^2 \tilde{G}_{0,β} }{∂ J^2}$ is not
  identically zero.

  It follows that in the coordinates $(k,t,J,ρ)$, the hamiltonian
  $\tilde{H}_{ε,β} = k + \tilde{G}_{ε,β} = k + ε \bar {G}_{0,β}(J) +
  O(ε^2)$ constructed in remark~\ref{rem:symplectic-isotopy} is in the
  form described by \cite[eq. 2.2.7]{MR0170705}. Proposition II in
  chapter II, \S 2 of~\cite{MR0170705} implies that if
  $\tilde{G}_{ε,β}$ is real-analytic, then there exists a positive
  measure set of invariant circles of $\tilde{ψ}_{ε}$ in a
  neighbourhood of the point $(I_0,ξ_0)$.

  The hypothesis of real-analyticity may be relaxed to infinite
  differentiability of $\tilde{H}_{ε,β}$ due to the work of
  Herman/F{é}joz~\cite[Th{é}or{è}me 57]{MR2104595}. To obtain the
  result in $C^s$, one may follow the arguments in the proof
  of~\cite[Theorem 2]{MR2299758} and observe that in the angle-action
  coordinates $(ρ,J)$, the Poincar{é} return map
  $\tilde{ψ}_{ε}=(ρ,J+2 π ε \bar{G}_{0,β}'(J)) +
  O(ε^2)$. By~\cite[Theorem 2.11]{MR1829194} and the subsequent
  discussion, if $\tilde{ψ}_{ε}$ is $C^s$ for $s>4$, then the present
  theorem follows provided that $r$ is chosen correctly. Finally, by
  remark~\ref{rem:differentiability} and
  lemma~\ref{lem:symplectic-poincare-maps}, if $H$, $G$ and $Τ$ are
  $C^r$ in the coordinates $(p,q,ξ)$ on the extended phase space
  $\cotangent Σ × \R$, then $\tilde{ψ}_{ε}$ is $C^{r-3}$,
  so $r$ must exceed $7$.
\end{proof}

\begin{remark}
  \label{rem:about-degenerate-kam-theory}
  The proof of Theorem \ref{thm:main-thm} is similar to the proofs in
  \cite{MR2299758,MR2519685}. The distinctive aspect here is that the
  coarse, topological, properties of the thermostat vector field $Τ$
  (it tries to heat the mechanical system $H$ when $H$ is small, and
  tries to cool when $H$ is large) combined with the existence of a
  Poincar{é} section and the invariant Gibbs-Boltzmann measure force
  the existence of KAM tori. The version of (degenerate) KAM theory
  used here was pioneered by Arnol$'$d in his study of the $n+1$ body
  problem. In this variant, the unperturbed hamiltonian is integrable
  with invariant isotropic tori (``fast angles'') and the leading term
  in the perturbation supplies a complementary set of ``slow angles''
  after averaging over the fast angles~\cite[chapter I, \S
  8]{MR0170705}. Non-degeneracy conditions similar to the usual ones
  (e.g. the frequency map is a local diffeomorphism) play a similar
  role in the theory. While this aspect of the theory is solid,
  Arnol$'$d's application to the $n+1$-body problem was discovered to
  be flawed by Herman in the late 1990s. This stimulated the extension
  of the Nash-Moser approach along with weaker conditions on the
  frequency map to this setting as exposed by
  F{é}joz~\cite{MR2104595}. Subsequently, Chierchia \& Pusateri and
  Chierchia \& Pinzari improved on Arnol'd's real-analytic
  results~\cite{MR2505319,MR2684064,MR2836051}.
\end{remark}

\section{Examples}
\label{sec:ex}

In this section, Theorem \ref{thm:main-thm} is applied to a number of
single thermostats that appear in the literature. Throughout, it will
be assumed that $H$ is \defn{well-behaved}.

\subsection{Separable Thermostats}
\label{sec:separable-thermostats}

This section describes an abstract type of thermostat vector field
that satisfies the properties of \ref{def:thermostat}.
\begin{definition}[Separable Thermostat]
  \label{def:separable-thermostat}
  Let
  \begin{align}
    \label{eq:separable-thermostat}
    Τ &= A\, ∂_q + B\, ∂_p + C\, ∂_{ξ}, &&\textrm{ where}\\\notag
    A &= A_0(q,p)\, A_1(ξ),& 
                             B &= B_0(q,p)\, B_1(ξ), & C &= C_0(q,p)\, C_1(ξ).
  \end{align}
  If $\D{μ}_{β}$ \eqref{eq:dmu} is invariant for $Τ$ and if there is
  an interval of regular values $[c_-,c_+]$ for $H$ such that on a
  connected component $e$ of $H^{-1}([c_-,c_+])$,
  \begin{enumerate}
  \item\label{cond:1} $\mean{A_0 H_q}$ and $\mean{B_0 H_p}$ are both non-negative and at least one is positive;
  \item\label{cond:2} $\mean{C}_0$ changes sign; and
  \item\label{cond:3} both $A_1$ and $B_1$ are odd functions of $ξ$;
  \item\label{cond:4} $C_1$ is a positive function.
  \end{enumerate}
  then $Τ$ is called a \defn{separable thermostat} for $H$.
\end{definition}

\begin{theorem}
  \label{thm:separable-thermostat}
  Let $Τ$ be a separable thermostat for $H$ and assume, in addition,
  that on the same component $e$ of $H^{-1}([c_-,c_+])$,
  \begin{enumerate}
  \item\label{thm:separable-thermostat-cond1} $\left( \sign{A_1} + \sign{B_1} \right) \sign{ξ} < 0$;
  \item\label{thm:separable-thermostat-cond2} $\degree{\mean{C}_0|e}=1$;
  \item\label{thm:separable-thermostat-cond3} the period of $\bar{R}_0$ is not constant.
  \end{enumerate}
  Then, the conclusions of Theorem \ref{thm:main-thm} and Corollary
  \ref{cor:non-ergodcity} apply.
\end{theorem}

\begin{proof}
  Hypothesis~\ref{thm:separable-thermostat-cond2} of the theorem
  combined with condition~\ref{cond:2} of
  definition~\ref{def:separable-thermostat} implies that $\mean{C}_0$
  changes from negative to positive. Therefore the stated hypotheses
  imply conditions (1--2) of lemma \ref{lem:ave-crit-pt-index} hold
  with $[d_-,d_+]=[-1,1]$ and the sign terms $Β=-1$, $Γ=1$.
\end{proof}

\let\Gamma=\oldGamma
\subsection{Nos{é}-Hoover Thermostat}
\label{sec:hoover-separable}
In this case~\cite{hoover}, the thermostat vector field is separable
and
\begin{align}
  \notag
  G_{β} &= H + ½ ξ ²,  &&&&& A_0,A_1 &= 0,\\
  \label{eq:hoover-separable}
  B_0 &= p, & B_1 &= -ξ, & C_0 &= p · H_p - T, & C_1 &= 1.
\end{align}
If one lets $Κ = \mean{B_0 H_p} = \mean{p · H_p}$ (twice the averaged
kinetic energy), then $\mean{C} = \bar{c} = Κ - T$ and
$H_I \bar{b} = - Κ\, ξ$. Note that a straightforward change of
variables~\cite[eq. 18]{intech} shows that
\begin{equation}
  \label{eq:nh-mean-ke}
  Κ = \mean{p · H_p} = \dfrac{1}{2 π} \oint p \ddt[q]{t} \D{θ} = I · H_I.
\end{equation}
(Compare to~\cite[eqs 34--5]{MR2519685}.) From
lemma~\ref{lem:ave-transformation}, it follows that the hamiltonian of
the vector field $\bar{R}_0$ is
\begin{equation}
  \label{eq:nh-g0b}
  \bar{G}_{0,β} = \left( β Z(β) \right)^{-1}\, I \, \exp \left( -β G_{β} \right).
\end{equation}
The integral found in \cite{MR2519685,PhysRevE.75.040102} can be
obtained by noting that $\bar{R}_0$ is also hamiltonian with respect
to the area form
$T ω_{β}/\bar{G}_{0,β} = H_I\, I^{-1}\, \D{I} ∧ \D{ξ}$. If one changes
to Darboux coordinates $(σ,ξ)$ where
\begin{equation}
  \label{eq:sigma-of-i}
  σ = \int_{I_0}^I H_I\, \D{I}/I
\end{equation}
and $(I_0,ξ_0=0)$ is the critical point (where $I_0 · H_I(I_0) = T$),
then the hamiltonian $\bar{G}_{0,T}$ of $\bar{R}_0$ with respect to
$\D{ξ} ∧ \D{σ}$ is
\begin{equation}
  \label{eq:nh-g0}
  \bar{G}_{0,T} = \underbrace{ξ ²/2}_{F(ξ)} + \underbrace{H(I(σ)) - T\, \ln I(σ)}_{U_T(σ)}.
\end{equation}
\begin{remark}
  \label{rem:comments-on-hoovers-thermostat-in-lit}
  A formula similar to \eqref{eq:nh-g0} appears in~\cite[eq.s 33--35,
  41]{MR2519685}, although the explicit form of $U_T$ is not there
  because the formula \eqref{eq:nh-mean-ke} was not used by the
  authors. A similar expression appears in
  \cite[p. 040102-3]{PhysRevE.75.040102} for the special case of the
  harmonic oscillator. \cite{MR2519685} defines the canonical
  coordinates $(ξ,σ)$ via $I = I_0\, e^{σ}$. This is due to an
  oversight: although the authors mention normalization of the return
  time to the Poincar{é} section, in their calculations the
  normalization is omitted with the consequence that the symplectic
  form used there omits the $H_I$ term. The reader may trace this back
  to the use of the vector field $\tilde{Y}_{ε}$ rather than the
  normalized vector field
  $\tilde{Z}_{ε}$~\eqref{eq:volume-time-rescaling}. Consequently, the
  right-hand side of the differential equations
  in~\cite[eq. 41]{MR2519685} need to be divided by
  $ω(a)$~\cite[eq. 37]{MR2519685} ($H_I$ in terms of the present
  paper). This is mentioned only because the proof of the following
  theorem is significantly less involved if one omits the $H_I$ term
  in the definition of $σ$~\eqref{eq:sigma-of-i}.
\end{remark}
\begin{lemma}
  \label{lem:nh-is-a-separable-thermostat}
  Let $H : \cotangent Σ → \R$ be well-behaved. Then, for all $T>0$,
  the Nos{é}-Hoover thermostat $Τ$ is a separable thermostat in the
  sense of definition~\ref{def:separable-thermostat}.
\end{lemma}
\begin{proof}
  All but condition~\ref{cond:2} is clear. To prove this, let's prove
  that $Κ$ surjects onto $[0,∞)$. Since $Κ$ maps the vertex set of
  $Γ_H$ to $0$, it suffices to prove $Κ$ is not bounded above.

  Assume that $Κ$ is bounded above by $T_1>0$ on $Γ_H$. Computation of
  the mean value of $Κ$ with respect to the Gibbs-Boltzmann measure
  $\exp(-β H) \D I \D{θ}/Z(β)$ on $Γ_H$ gives
  \begin{equation}
    \label{eq:nh-gibbs-mean-value}
    T_1 ≥ \gibbsexp{β}{Κ},
  \end{equation}
  where $T=1/β$. But since $\gibbsexp{β}{Κ} = \gibbsexp{β}{p · H_p}$
  and $H$ is well-behaved, by hypothesis,
  condition~\itref{it:h-is-well-behaved-unbounded-temp} of
  definition~\ref{def:h-is-well-behaved} implies the right-hand side
  of \eqref{eq:nh-gibbs-mean-value} is
  unbounded. Contradiction. Therefore, $Κ$ is not bounded above on
  $Γ_H$.

  Since $Γ_H$ has finitely-many compact edges and $Κ$ is continuous,
  $Κ$ must map one of the (at most two) non-compact edges onto
  $[0,∞)$.
\end{proof}
\begin{remark}
  \label{rem:nh-is-a-separable-thermostat}
  One can compute
  $\gibbsexp{β}{Κ} = T × \left(1 - \limsup_{H → ∞} I \exp(-β H)
  \right)$, which shows that it suffices to assume that
  $I = o(e^{β H})$ for all $β>0$, rather than
  condition~\ref{it:h-is-well-behaved-unbounded-temp} of
  definition~\ref{def:h-is-well-behaved}.
\end{remark}

\begin{theorem}
  \label{thm:nh-kam-nec}
  Let $H : \cotangent Σ → \R$ be a well-behaved, real-analytic
  hamiltonian. Then the period of $\bar{R}_0$ is not constant in a
  neighbourhood of the critical point $(I_0,0)$.
\end{theorem}
\begin{proof}[Proof of Theorem~\ref{thm:nh-kam-nec}]
  Since $H$ is bounded below, it can be assumed without loss that $0$
  is its (and $V$'s) global minimum. Let $H_0 = H(I_0)$ be the value
  of $H$ at the thermostatic equilibrium action $I_0$.
  
  Assume that the period of $\bar{R}_0$ is constant in a neighbourhood
  of the elliptic critical point $(I_0,0)$. Real analyticity implies
  that the period is constant on the entire phase space. Since $(σ,ξ)$
  are Darboux coordinates and $\bar{G}_{0,T}$ is a mechanical
  hamiltonian, the potential $U=U_T$ must be \textem{isochronous} and
  real analytic. In particular, $U$ must increase to the right of
  $I_0$ and decrease to the left.
  
  From this, it follows that $U$ is bounded below, with a unique
  critical point at $σ = 0$. Moreover,
  \begin{align}
    \label{eq:u_s-u_ss}
    U_{σ} &= I - T I_H = I \left( 1 - T/Κ(I) \right) &&& U_{σ σ} &= I \left( I_H - T I_{HH} \right).
  \end{align}

  \begin{claim}
    \label{cl:1-no-crit-pts}
    $H$ has no critical values in $[H_0, ∞)$.
  \end{claim}

  We use the notation and terminology introduced in
  section~\ref{sec:prelims}. Let $γ_0 ∈ Γ_H$ be the connected
  component of $H^{-1}(H_0)$ where the thermostatic equilibrium is
  attained. Since $Κ(I_0)=T$, the point $γ_0$ is not a vertex, so let
  $c_0 ⊂ Γ_H$ be the interior of the edge containing $γ_0$. The
  claim~\ref{cl:1-no-crit-pts} is equivalent to the claim that $\sup
  c_0 = ± ∞$.

  Assume otherwise, so $v = \sup c_0$ is a saddle vertex. But
  $\lim_{γ ↗ v} \tilde{I}(γ)>0$ and $Κ(\tilde{I}(v)) = 0$ from
  section~\ref{sec:prelims}, which implies that $U_{σ} ↘ -∞$ as
  $γ ↗ v$. Since $U$ is increasing to the right of $σ=0$ (i.e. for
  $γ ∈ c_0$ such that $γ_0 < γ < v$), this is a
  contradiction. Therefore, the $\sup c_0$ is not a saddle vertex so
  it can only be $± ∞$.

  It is a tautology that the following dichotomy holds.

  \begin{definition}[Critical Dichotomy]
    \label{def:h-crit-dichotomy}
    Let $c_0 ∈ Γ_H$ be the interior of the edge described
    above. Either $\min c_0$ is
    \begin{enumerate}
    \item a saddle vertex; or
    \item a local minimum vertex.
    \end{enumerate}
  \end{definition}

  \begin{claim}
    \label{cl:2-consequences}
    In case (2) of the critical dichotomy the minimum is degenerate.
  \end{claim}
  
  Since $Γ_H$ is connected and $\sup c_0 = ± ∞$, one concludes that
  $\min c_0$ must be the unique global minimum vertex of $Γ_H$. Since
  $\tilde{I}(γ) ↘ 0$ and $H ↘ 0$ as $γ ↘ \min c_0$, therefore $U ↗
  ∞$. If $\liminf U_{σ} = -∞$, then \eqref{eq:u_s-u_ss} implies that
  $\limsup I_H = ∞$ ($\liminf H_I = 0$) and so the minimum must be
  degenerate.

  Therefore, to complete the proof of claim \ref{cl:2-consequences},
  it remains to show that $\liminf_{I → 0^+} U_{σ} = -∞$. Assume that
  $U_{σ}$ is bounded below by $-1/c$ for some $c>0$. Let $σ_{±}(u)$ be
  the inverses of $U(σ)$: $U(σ_{±}(u))=u$ for all $u ≥ 0$ and
  $± σ_{±}(u) > 0$ for all $u>0$. Then $σ_{-}'(u)$ is bounded above by
  $-c$, so there is a $d>0$ such that $σ_{-}(u) ≤ - cu + d$ for all
  $u$. Since $σ_{+}(u) ≥ 0$, it follows that
  $Δ(u) = σ_{+}(u) - σ_{-}(u) ≥ cu-d$ for all $u > 0$. But, it is
  known~\cite{MR0120782,17462920030101BM} that if $U$ is an isochronous
  potential, then $Δ(u)$ is a constant multiple of $\sqrt{u}$. This
  contradiction proves that $\liminf U_{σ}=-∞$, as required.

  \begin{claim}
    \label{cl:1-consequences}
    In case (1) of the critical dichotomy the following holds:
    \begin{enumerate}
    \item[] As $γ ↘ \min c_0$,
    \item\label{cl:1} $I ↘ I_1 > 0$;
    \item\label{cl:3} $I_H ↗ ∞$;
    \item\label{cl:2} $U_{σ} ↘ -∞$. And,
    \item\label{cl:4} $U$ is bounded.
    \end{enumerate}
  \end{claim}

  \ref{cl:1} \& \ref{cl:3} follow from section~\ref{sec:prelims} since
  $\min c_0$ is a saddle vertex. \ref{cl:2} follows from
  \eqref{eq:u_s-u_ss} and \ref{cl:3}. From \ref{cl:1} and
  \eqref{eq:nh-g0} it follows that $U$ is bounded from above on the
  interval $(\min c_0, γ_0]$ and \ref{cl:2} implies that $U$ does not
  extend to the left of $\min c_0$; on the other hand, $U$ reaches its
  minimum value at $γ_0$ (where $σ=0$). But the image of
  $[γ_0,\sup c_0)$ equals that of $(\min c_0, γ_0]$. Therefore, $U$ is
  bounded, hence \ref{cl:4}.
  
  Finally, the finite value $g = \sup U_T(σ)$ is a critical point of
  $\bar{G}_{0,T}$ since the topology of the level sets of
  $\bar{G}_{0,T}$ change at height $g$. This contradicts
  the constancy of the period of $\bar{R}_0$.

  Therefore, only case (2) of the critical dichotomy can hold, in
  which case claim~\ref{cl:2-consequences} holds, too.
\end{proof}

\begin{remark}[Isochronous potentials]
  \label{rem:isochronous-potentials}
  The proof of claim~\ref{cl:2-consequences} uses the fact that if $U$
  is an isochronous potential, then $Δ(u)$ is proportional to
  $\sqrt{u}$~\cite{MR0120782,17462920030101BM}. A more elementary
  proof consists in showing that the action function $J$ satisfies
  $2 π J(g) ≥ c g^{\frac{3}{2}} - 2d g^{½}$ while $J$ is proportional
  to $g$.
\end{remark}

\begin{remark}[Degenerate global minimum]
  \label{rem:nh-degenerate-crit-pt-at-0}
  In case (2) of the critical dichotomy, above, it is shown that the
  critical point at $H=0$ must be degenerate. Since $H$ is real
  analytic, the degeneracy is finite, so a straightforward calculation
  shows that $I=c H^{r} + o(H^{r})$ for some constant $c>0$ and
  integer $n ≥ 2$ with $r=½ (1+1/n)$. Therefore, the infimum $σ_1$ of
  $σ$ is finite.

  Let's examine a particular example: $V(q) = (ω q)^{2n}$ for $n ≥
  2$. A computation shows that $I = γ\, H^r$ where $r=½ (1+1/n)$ and
  $γ$ is a structural constant, $H_0 = r T$, $σ = (H^s - H_0^s)/(s γ)$
  where $s=1-r$ (so $σ_1 = - H_0^s/(s γ)$), and
  \[ U/H_0 = \left( 1 + σ/σ_1 \right)^{1/s} - \ln \left( 1 + σ/σ_1
    \right)^{1/s}. \] The third-order Birkhoff normal form of
  $\bar{G}_{0,T}$ is, up to a $T$-dependent constant,
  \begin{align*}
    & ω \, γ \, J - A \, γ² \, J^2 + B \, γ³ J ³ + O(J^4), &&& \text{where} && ω &= H_0^{1/2-s},\, A = \frac{6\,s^2 - 6\,s + 1}{{24\,H_0^{2s}}},\\
    &                                                      &&& \textrm{and} && B &= \frac{180\,s^4 - 312\,s^3 + 168\,s^2 - 36\,s + 5}{1728\,H_0^{3s+1/2}}
  \end{align*}
  Since the resultant of the numerators of $A$ and $B$ is $-6912$, the
  hamiltonian is not isochronous for any value of $s$, hence $n$.

  \medskip

  On the other hand, one can try to ``design'' a hamiltonian $H$ given
  $U$ using \eqref{eq:sigma-of-i} \& \eqref{eq:nh-g0}. In this case,
  one gets
  \begin{equation}
    \label{eq:h-in-terms-of-u}
    H(σ) = H_0 - T \ln \left( 1 - β \, I_0 \, \int_0^{σ} \exp(-β \tilde{U}(σ)) \, \D{σ} \right),
  \end{equation}
  where $\tilde{U} = U-U_0$ is a function defined on $(σ_1,∞)$;
  $U_0=U(0)=H_0 - T \ln I_0$. The function $\tilde{U}$ satisfies
  \begin{align}
    \label{eq:tu-a}
    \int_{σ_1}^{∞} \exp \left( -β \tilde{U}(σ) \right) \D{σ} &= \frac{T}{I_0} × \exp(β H_0), \\
    \label{eq:tu-b}
    \int_0^{∞} \exp \left( -β \tilde{U}(σ) \right) \D{σ}   &= \frac{T}{I_0},
  \end{align}
  so the integral of $\exp(-β U)$ over $(σ_1,∞)$ is $T$.

  One can use equations \eqref{eq:h-in-terms-of-u}--\eqref{eq:tu-b} to
  reconstruct $H$ and $I$ as functions of $σ$. A tractable case is
  when $\tilde{U}$ is rational; since the isochronous case is of
  particular interest one can assume, without significant loss of
  generality, that
  \begin{equation}
    \label{eq:tu-rational}
    \tilde{U}(σ) = \left( (σ+1) - (σ+1)^{-1} \right)^2.
  \end{equation}
  In this case, one computes
  \begin{align}
    \label{eq:tu-a-ex1}
    \int_{-1}^{∞} \exp \left( -β \tilde{U}(σ) \right) \D{σ} &= \sqrt{\frac{π}{4 β}}\\
    \label{eq:tu-b-ex1}
    \int_{0}^{∞} \exp \left( -β \tilde{U}(σ) \right) \D{σ} &= \sqrt{\frac{π}{4 β}} \, W_0, &&& W_0 &= W(0),
  \end{align}
  where $W=W(σ)$ is defined in \eqref{eq:h-i-of-sigma}. Combined with
  \eqref{eq:tu-a} \& \eqref{eq:tu-b}, one finds that
  \begin{align}
    \label{eq:h-of-t-ex1}
    H_0 &= -T \, \ln(W_0), &&& I_0 &= \frac{2\,T^{½}}{\sqrt{π} \, W_0}.
  \end{align}
  From this one determines that (with $σ=τ-1$)
  \begin{align}
    \label{eq:h-i-of-sigma}
    H &= -T \ln W, &&& I = H_{σ} &= \frac{2\,T^{½}}{\sqrt{ π } \, W} × \exp \left( -β (τ-1/τ)² \right) \\\notag
      &\text{where} &&& 2W       &= e^{4 β} \, \erfc \left( \sqrt{β} (τ+1/τ)\right) + \erfc \left( \sqrt{β} (τ-1/τ) \right),
  \end{align}
  and $\erfc$ is the complementary error function~\cite[Chapter
  7]{MR1225604}. One finds that $\lim_{I → 0^+} H/I^{α}$ is $0$ if
  $α=1$ and $∞$ if $α > 1$. On the other hand, for the oscillator with
  potential $V=(ω q)^{2n}$ the same limit is $0$ for $1 ≤ α < 1/r$ and
  $1/r > 1$ when $n > 1$. So, by the arguments of the first paragraph
  of this remark, the hamiltonian $H$ cannot be mechanical and real
  analytic. Indeed, I am uncertain if $H$ extends to a $C^2$
  hamiltonian, mechanical or otherwise, in the $(q,p)$ coordinates.

  To rule out all possibilities for $H$, one needs to know the general
  form of an isochronous potential. Bolotin \&
  MacKay~\cite[p. 220]{17462920030101BM} prove that if $U$ is a $C^r$
  ($r ≥ 2$) isochronous potential on $\R$ with a critical point at
  $0$, and $U''(0)=2$, then there is a continuous (shear) function
  $\hat{σ} : [0,∞) → \R$ such that $\hat{σ}(u) = o(\sqrt{u})$ at $0$,
  $\hat{σ}$ is $C^r$ on $(0,∞)$, $|\hat{σ}'(u)| < ½ u^{-\frac{1}{2}}$
  for $u>0$ and the function $U=U(σ)$ is defined implicitly as the
  zero locus of the function $K(σ,u) = u - (σ - \hat{σ}(u))
  ²$. Geometrically, the graph of the potential $U$ is obtained from
  the parabola $u=σ ²$ by the shearing transformation
  $(σ,u) → (σ - \hat{σ}(u),u)$. The general case is obtained from this
  result by rescaling $u$ and restricting the domain of $K$. However,
  even with this additional information, I am unable to prove if the
  example with $\tilde{U}$ determined by \eqref{eq:tu-rational} is
  illustrative of the general case or a singular case.
\end{remark}

\begin{remark}[The averaged kinetic energy]
\label{rem:nh-remark-on-ave-ke}

As mentioned above, the formula for the averaged kinetic
energy~\eqref{eq:nh-mean-ke} is not used in
\cite{MR2519685}. Nonetheless, the authors numerically compute this
integral, called $k_0$ there, for the planar pendulum
$H = ½ p ² - \cos q$ by integrating Hamilton's equations. On the
other hand, it follows from \eqref{eq:nh-mean-ke} and \cite[eq.s
11--13]{intech} that
\[ Κ = \dfrac{2(H+1)}{1 - k K'(k)/K(k)}, \qquad H+1 = \dfrac{2}{k^2}, \quad I = \dfrac{8 K(k)}{π k}, \]
where $K(k)$ is an elliptic integral described in
\cite{intech}. Figure \ref{fig:llm-t-pendulum}
(resp. \ref{fig:llm-w-pendulum}) reproduces figure 10 (resp. 11) of
\cite{MR2519685} using only these formulas.

\begin{figure}[h]
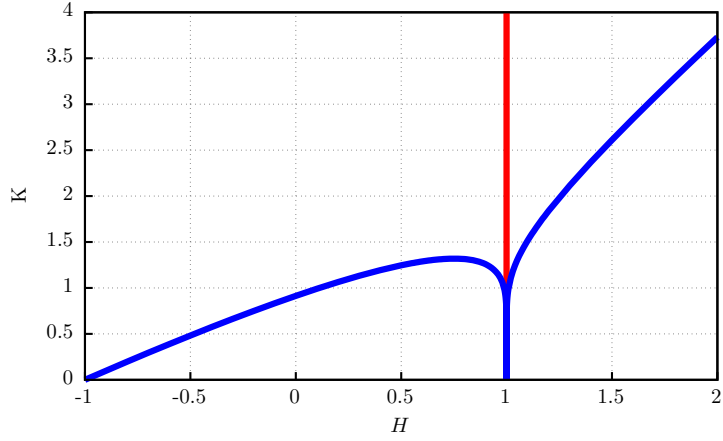

  \caption{$Κ$ vs. $H$ for the planar pendulum $H = ½ p ² - \cos q$, c.f. \cite[Figure 10]{MR2519685}. The vertical red line is the energy level of the saddle fixed point.}
  \label{fig:llm-t-pendulum}
  \ltxfigure{llm-t-pendulum.tex}{10cm}{!}
\end{figure}

\begin{figure}[h]
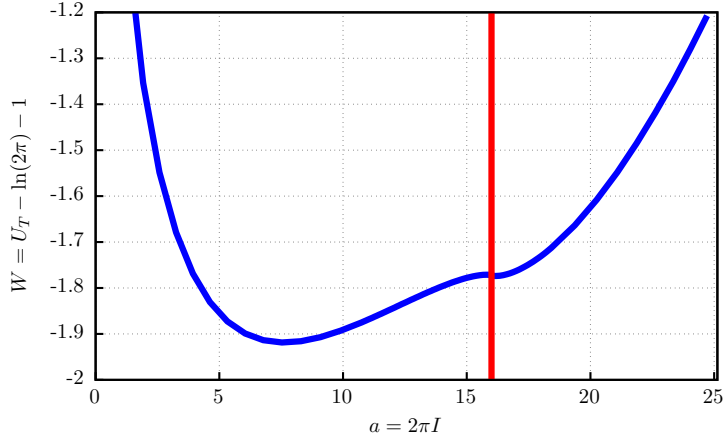

  \caption{$W$ vs. $a$ for the planar pendulum $H = ½ p ² - \cos q$, c.f. \cite[Figure 11]{MR2519685}. The vertical red line is the action of the saddle fixed point.}
  \label{fig:llm-w-pendulum}
  \ltxfigure{llm-w-pendulum.tex}{10cm}{!}
\end{figure}

\begin{figure}[h]
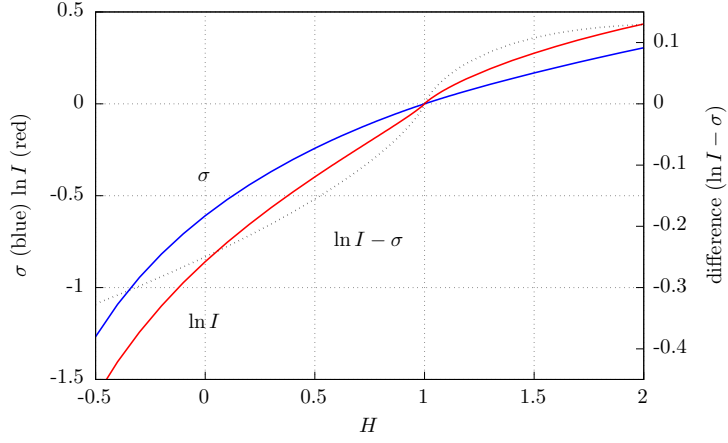

  \caption{$σ$, $\ln I$ and their difference for the planar pendulum.}
  \label{fig:llm-logi-v-sigma-pendulum}
  \ltxfigure{llm-logi-v-sigma-pendulum.tex}{10cm}{!}
\end{figure}
\end{remark}

\subsection{Logistic Thermostat of Tapias, Bravetti \& Sanders}
\label{sec:tap-brav-san}

In~\cite{10.12921/cmst.2016.0000061}, Tapias, Bravetti \& Sanders
introduce a ``logistic thermostat'', which is very similar to Hoover's
with:
\begin{align}
  \notag
  G_{β} &= H + F, &&&&& A_0,A_1 &= 0,\\
  \label{eq:logistic-separable}
  B_0 &= p, & B_1 &= -F', & C_0 &= p · H_p  - T, & C_1 &= 1,
\end{align}
where $F = \ln \left( \cosh \left( ξ \right)\right)$ and
$F' = \tanh \left( ξ \right)$. \cite{10.12921/cmst.2016.0000061} uses
the variable $ζ$ for the thermostat and parameter $Q$ (which would be
$\sqrt Q$ in Hoover's thermostat). These are related to $ξ$ and $ε$ by
$ξ = ε ζ/2$ and $ε = 1/Q$. That paper also uses $2 T F$ in lieu of
$F$, but there is no qualitative difference in the analysis below. The
choices here are dictated by the desire to portray clearly this
thermostat as a perturbation of Hoover's (which it is), since with the
choices made, $F = ½ ξ ² + O(ξ ⁴)$.

Indeed, let $F_{η}(ξ) = η^{-2} F(η ξ) = ½ ξ ² + η^2
\tilde{F}_{η}(ξ)$. For $η=0$, one has Hoover's thermostat and for
$η=1$, the logistic thermostat. Equations
\eqref{eq:nh-g0b}--\eqref{eq:nh-g0} hold for the averaged thermostat
from equations \eqref{eq:logistic-separable} (with $½ ξ ²$ replaced by
$F(ξ)$ in the latter equation). Let $φ_{η} : (ρ, J) → (ξ,σ)$ be the
angle-action map for $\bar{G}_{0,T; η}$ where $T$ is held fixed. The
thermostatic equilibrium is at $(σ,ξ)=(0,0)$ independently of $η$. If
$H$ is real analytic, then $η \mapsto \bar{G}_{0,T; η}''(J)$ is a
real-analytic function that does not vanish at $η=0$ for some
$J$. Therefore, it is non-zero for all but a countable set without
accumulation points. This ``almost'' proves that the averaged hamiltonian
is never isochronous.

\begin{theorem}
  \label{thm:logistic-nec-kam}
  Let $H : \cotangent Σ → \R$ be a well-behaved, real-analytic
  hamiltonian. If $F$ \eqref{eq:logistic-separable} is even,
  real-analytic and satisfies $F(ξ)/|ξ| → c>0$ as $|ξ| → ∞$, then the
  period of $\bar{R}_0$ is not constant in a neighbourhood of the
  critical point $(I_0,0)$.
\end{theorem}
\begin{remark}
  \label{rem:logistic-nec-kam}
  It is clear that $F(ξ) = \ln \cosh(ξ)$ satisfies the hypothesis of
  this theorem. The strategy of the proof is similar to that employed
  for Theorem~\ref{thm:nh-kam-nec}, with a similar result: if
  $\bar{R}_0$ has a constant period, then $H$ has a single, degenerate
  minimum point.
\end{remark}

\begin{proof}
  Assume the stated hypotheses. Condition \ref{cond:3} of definition
  \ref{def:separable-thermostat} requires that $F$ is even, so
  $F'(0)=0$ and without loss, $F(0)=0$. If $U_T(I(σ))$ is bounded
  above, $γ=\sup \set{ U_T }$ is a critical value of
  $\bar{G}_{0,T}$. If $U_T$ is not constant, then $\bar{G}_{0,T}$ has
  two distinct critical values, which contradicts
  isochronicity. Therefore, $U_T$ is unbounded from above and has only
  one critical value (an absolute minimum, which may be assumed to be
  $0$) at the point $σ=0$.

  Claim~\ref{cl:1-no-crit-pts} holds here, with the same proof as
  above.

  To prove that claim~\ref{cl:2-consequences} holds here, it suffices
  to prove that $\liminf U_{σ} = -∞$. To do this, one uses a method
  similar to that mentioned in
  remark~\ref{rem:isochronous-potentials}.

  For $g>0$, let $σ_{±}(g)$ be the local inverse to $U_T$:
  $g = U_T(I(σ_{±}(g)))$ and $± σ_{±}(g) > 0$. Let
  $Δ = σ_+ - σ_-$. From the previous paragraph it follows that both
  $σ_+$ and $-σ_-$ are increasing, so $Δ$ is, too.

  Let $g \gg 1$. By hypothesis, $F(ξ) = c\, |ξ|\, (1+k(ξ))$ where
  $k(ξ) → 0$ as $|ξ| → ∞$; without loss of generality, it is assumed
  $c=1$. The level set $\set{ \bar{G}_{0,T} = g }$ satisfies
  $|ξ| = g - U_T(I(σ)) + o(g)$ for $|ξ| \gg 1$. Let $(ρ,J)$ be the
  angle-action variables for $\bar{G}_{0,T}$. It follows, by comparing
  inscribed and circumscribed rectangles, that for any $0 < α < 1$,
  \begin{align}
    \label{eq:logistic-i-bd}
    2 (1-α) g Δ (α g)  &≤ 2 π J + o(g) ≤ 2 g Δ(g).
  \end{align}
  If $\bar{R}_0$ has a constant period in a neighbourhood of the
  critical point $(I_0,0)$, then since $F$ and $H$ are real-analytic,
  its period is constant. Therefore, $\bar{G}_{0,T} = ω J$ is linear
  in $J$, and
  \begin{align}
    \label{eq:logistic-limsup-inf}
    \frac{π}{ω} &≤ \lim_{g → ∞} Δ(g) ≤ \frac{π}{ω(1-α)},
  \end{align}
  hence $Δ(g) ↗ π/ω$ as $g → ∞$. This implies that $σ_-(g)$ has an
  infimum $σ_1$ and $σ_+(g)$ has a supremum $σ_2 = σ_1+π/ω$.

  To summarize: $U=U_T(I(σ))$ has a bounded domain $(σ_1,σ_2)$ and
  diverges to $∞$ as $σ$ approaches either endpoint. Therefore $U_{σ}$
  diverges to $-∞$ at $σ_1$ (resp. $∞$ at $σ_2$). This proves
  claim~\ref{cl:2-consequences} holds here, too.

  To prove that claim~\ref{cl:1-consequences} holds here, one notes
  that the proof above does not make use of the mechanical nature of
  $\bar{G}_{0,T}$, only its quasi-mechanical nature. Since~\ref{cl:4}
  of that claim is that $U$ is bounded, one obtains a contradiction.

  Therefore, only case (2) of the critical dichotomy can hold, in
  which case claim~\ref{cl:2-consequences} holds, too.
\end{proof}

\begin{remark}[Degenerate global minimum]
  \label{rem:logistic-degenerate-crit-pt-at-0}
  Similar to remark~\ref{rem:nh-degenerate-crit-pt-at-0}, one can
  examine the case where the potential energy has a degenerate
  critical point at $q=0$, e.g. $V=(ω q)^{2n}$. Similar calculations,
  using the $4$-th order Birkhoff normal form of
  $\bar{G}_{0,T;η}$, show that the isochronicity condition is never
  satisfied.
\end{remark}

\subsection{Watanabe \& Kobayashi}
\label{sec:wat-kob}

In~\cite{PhysRevE.75.040102}, Watanabe \& Kobayashi generalize
Hoover's thermostat by setting
\begin{align}
  \notag
  G_{β} &= H + ½ ξ ², &&&&& A_0,A_1 &= 0,\\
  \label{eq:wat-kob-separable}
  B_0 &= p^k, & B_1 &= -ξ^l, & C_0 &= p^{k-1} \left( p · H_p  - k \, T \right), & C_1 &= ζ_l(ξ)
\end{align}
where, when $l=2n+1$, $ζ_l$ is the $n$-th Maclaurin polynomial of
$(2/β)^n\, n!\, \exp(x)$ evaluated at $x = β ξ ²/2$~\cite[eq.s
8--14]{PhysRevE.75.040102}. For $(k,l)=(1,1)$, one has Hoover's
thermostat.

In order to have conditions (1--4) hold in
definition~\ref{def:separable-thermostat} and (1--2) in
Theorem~\ref{thm:separable-thermostat}, one needs both $k$ and $l$ to
be odd: $k=2m+1$, $l=2n+1$. This is assumed in
\cite{PhysRevE.75.040102}. It is also assumed there that
$H = (q ² + p ²)/2$, but this is not necessary.

The only challenge is condition \ref{cond:2}: to locate an interval of
regular values $[c_-,c_+]$ such that $\mean{C}_0$ alternates sign. To
do this, let us define
\begin{align}
  \label{eq:wk-c0}
  f(I) &= k \, \mean{p^{k-1}}, & \tilde{Κ}(I) &= \frac{\mean{ p^k · H_p }}{f(I)}, &
  \mean{C}_0 &= f(I)\, \left( \tilde{Κ}(I) - T\right)
\end{align}
where $f (=f_k)$ and $\tilde{Κ} (=\tilde{Κ}_k)$ are smooth, positive
functions of $I$. $\tilde{Κ}$ can be viewed as a weighted average
temperature along an orbit similar to that defined in
\eqref{eq:nh-mean-ke}--indeed, when $k=1$, one recovers the definition
of \eqref{eq:nh-mean-ke}.

\begin{lemma}
  \label{lem:wk-mech-implies-k-surjective}
  If $H : \cotangent Σ → \R$ is a well-behaved, smooth hamiltonian,
  then $\tilde{Κ} : B_H → \R$ surjects onto $[0,∞)$ less a finite
  number of points.
\end{lemma}

\begin{proof}
  If $k=1$, then this is clear, so assume $k ≥ 3$. First, let's
  observe that $\tilde{Κ}_k$ is not bounded above. If $\tilde{Κ}_k$
  were bounded above by some $T_1 > 0$, then integrating over $Γ_H$
  against the Gibbs-Boltzmann measure $\exp(-β H)\, \D{I}/Z(β)$ implies
  that
  $$\gibbsexp{β}{p^k · H_p} ≤ k T_1 \gibbsexp{β}{p^{k-1}}$$ where
  $\gibbsexp{β}{φ} = \iint_{\cotangent Σ} φ\, \exp(-β H)\, \D{p}
  \D{q}/Z(β)$ is the mean-value of $φ$ with respect to the
  Gibbs-Boltzmann probability measure at temperature $T=1/β$. However,
  this bound contradicts
  condition~\itref{it:h-is-well-behaved-unbounded-temp} of the
  definition~\ref{def:h-is-well-behaved}.

  To proceed, let us define
  \begin{align}
    \label{eq:wk-fk}
    F(I) &= \dfrac{1}{2 π} \iint k\, p^{k-1}\, \D p\, \D q = \int f(I)\, \D I, 
  \end{align}
  where the integration is done over the region in the plane bounded
  by the component $γ_0$ of a level set of $H$, i.e.
  $\set{ γ ∈ Γ_H \mid γ ≤ γ_0 }$. A change of variables shows that
  \begin{align}
    \label{eq:wk-pkhp}
    \mean{p^k · H_p} &= H_I × F(I), \\\label{eq:wk-tildek}
    \tilde{Κ} &= H_I\, F\, /F_I = F/F_H.
  \end{align}

  \begin{claim}
    \label{cl:wk-mech-implies-k-surjective}
    The following alternatives hold at a critical vertex $v$ in the
    edge $e$.
    \begin{enumerate}
    \item\label{it:cl:wk-mech-implies-k-surjective-1} If $v$ is a
      local minimum vertex, then as $γ ↘ v$, $F/F_H → 0$ and
      $\tilde{Κ} → 0$;
    \item\label{it:cl:wk-mech-implies-k-surjective-2} if $v$ is a
      saddle vertex and all critical points in $v$ lie in $\set{p=0}$,
      then as $γ → v$, $F$ and $F_H$ are bounded away from $0$ and
      $\tilde{Κ} → κ_v > 0$;
    \item\label{it:cl:wk-mech-implies-k-surjective-3} if $v$ is a
      saddle vertex that contains a critical point $x_c=(q_c,p_c)$
      with $p_c ≠ 0$, then as $γ → v$, both $F$ and $F_I$ are bounded
      away from $0$, and $\tilde{Κ} → 0$.
    \end{enumerate}
  \end{claim}
  \begin{proof}[Proof of claim]
    In case~\itref{it:cl:wk-mech-implies-k-surjective-1}, let the
    local minimum value of $H$ be $H_0$ and without loss assume
    $H_0=0$. From the non-degeneracy of the local minimum, it follows
    that there is a positive constant $b$ such that
    $F = b H^{(k+1)/2} + O(H^{k/2+1})$. Therefore,
    $F/F_H = \frac{2}{k+1} H + O(H^{3/2}) → 0$ as the critical point
    is approached.

    In case~\itref{it:cl:wk-mech-implies-k-surjective-3}, one has that
    the invariant probability measure $\D{θ}/2 π$ on the cycle $γ$
    converges in the weak-$\star$ topology to an invariant probability
    measure that is supported on critical points $(q_i,p_i)$ that are
    limits of points on $γ$ as $γ → v$. This implies that $f=F_I$
    converges to $\sum_i λ_i p_i^{k-1}$ where $λ_i > 0$ and
    $\sum_i λ_i=1$. Therefore, both $F$ and $F_I$ are bounded away
    from $0$ and $∞$ in a neighbourhood of $v$ on the edge $e$. Since
    $v$ is a saddle, $H_I → 0$.

    In case~\itref{it:cl:wk-mech-implies-k-surjective-2}, the previous
    argument yields $F_I → 0$. On the other hand, it is
    straightforward to establish the existence of a lower bound $c>0$
    such that $Δ F ≥ c^{-1} Δ H$ in a neighbourhood of the critical
    cycles: replace the integrand $p^{k-1}$ by the censored integrand
    $Χ_{η}(q,p) × p^{k-1}$ where $Χ_{η}(q,p)$ is $0$ if $|p|<η$ and
    $1$ otherwise. From this it follows that
    $\liminf_{γ \to v} \tilde{Κ}$ is positive and since $\tilde{Κ}|_e$
    is continuous, the claim is established.
  \end{proof}
  
  \begin{figure}[h]
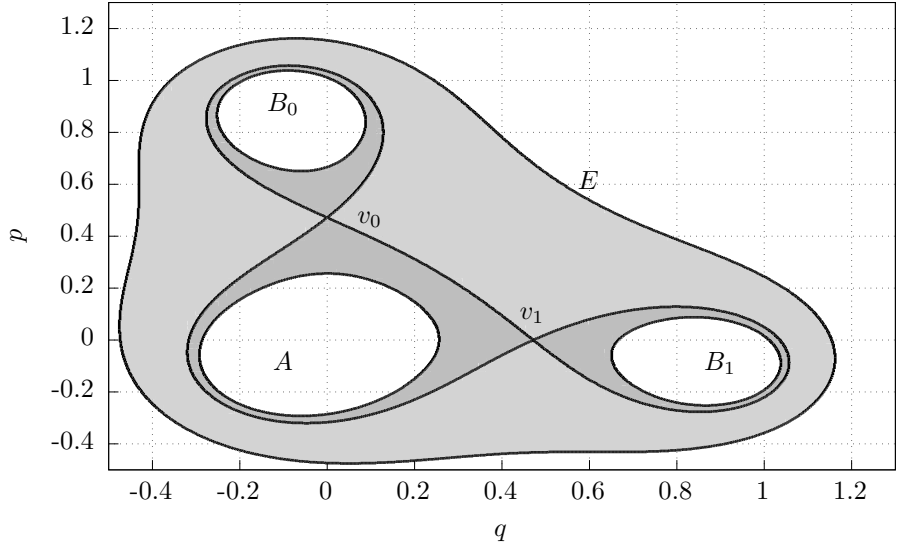

    \caption{A neighbourhood of a critical component.}
    \label{fig:wk-a-neighbourhood-of-a-critical-cycle}
    \ltxfigure{nhwc-wk1.ltx}{\textwidth}{!}
  \end{figure}
  
  To proceed with the proof of the lemma, let $e ⊂ Γ_H$ be an edge on
  the graph of $H$ and $v$ a vertex of $e$. In cases
  (\ref{it:cl:wk-mech-implies-k-surjective-1} \&
  \ref{it:cl:wk-mech-implies-k-surjective-3}) of
  claim~\ref{cl:wk-mech-implies-k-surjective}, $\tilde{Κ}$ vanishes at
  $v$. Otherwise, in case
  (\ref{it:cl:wk-mech-implies-k-surjective-2}), if $v=\min e$ is a
  saddle vertex of $H$, then there are incoming edges
  $e_1, \ldots, e_n$ with $\max e_i = \min e$ such each of the
  incoming edges satisfy (\ref{it:cl:wk-mech-implies-k-surjective-2}),
  too. Let $a_i$ (resp. $b_i$) be the limit of $F$ (resp. $F_H$) along
  $e_i$ as the contour approaches $\max e_i$. Then, the limit from
  above along $e$ of $\tilde{Κ}$ is
  $(a_1+\cdots+a_n)/(b_1+\cdots+b_n)$. This value lies between the
  minimum and maximum limits along the incoming edges. If the minimum
  and maximum are the same (so $\tilde{Κ}$ is actually single-valued
  and continuous at this vertex), then the common value must be
  removed from $[0,∞)$; otherwise, it lies in the image of $B_H$.

  To conclude, there is a finite set of edges in $Γ_H$ that connect a
  local minimum vertex to the maximal edge whose supremum is $∞$ and
  where at each vertex the limit from below of $\tilde{Κ}$ is at least
  as large as the limit from above. This implies that the image of the
  interiors of these edges contains $[0,∞)$ except possibly the finite
  set of limiting values of $\tilde{Κ}$ at vertices of these edges.
\end{proof}

\begin{definition}
  \label{def:wk-admissible-temperatures}
  Let $τ_H$ be the set $[0,∞)$ less the limiting values of
  $\tilde{Κ}_k$ at the vertex set of $Γ_H$. We call $τ_H$ the set of
  admissible temperatures.
\end{definition}

\begin{lemma}
  \label{lem:wk-condition-2-is-almost-always-satisfied}
  Let $T ∈ τ_H$. Then, there is an edge $e ⊂ Γ_H$ and a closed
  interval $l ⊂ e$ such that
  \begin{align}
    \label{eq:wk-good-interval}
    \tilde{Κ}_k(\min l) < T < \tilde{Κ}_k(\max l)
  \end{align}
  and hence $T ∈ \tilde{Κ}_k(l)$.
\end{lemma}
\begin{proof}
  The previous lemma proved that there is a sub-graph $Γ'_H ⊂ Γ_H$ that
  is homeomorphic to $[0,∞)$ such that $\tilde{Κ}(Γ'_H) = [0,∞)$ and
  $\tilde{Κ}$ is continuous on $Γ'_H$ except possibly at the vertices;
  and it extends to an upper semi-continuous function at the
  vertices. Let
  $\tilde{Κ}^+(x) = \max\set{ \tilde{Κ}(y) \mid y ≤ x, y ∈ Γ'_H}$. The
  function $\tilde{Κ}^+$ is a non-decreasing, continuous map from
  $Γ'_H$ onto $[0,∞)$ and $\tilde{Κ}^+ ≥ \tilde{Κ}$ point-wise.

  Let $T ∈ τ_H$. Then, $S = \set{ x \mid \tilde{Κ}^+(x) = T}$ is a
  closed, connected subset of $Γ'_H$. Let $y = \max S$. Since $T$ is
  not in the image of a vertex in $Γ_H$, there is an edge $e$ such
  that $y ∈ e^o$. Since $e^o$ is open, there is a $z ∈ e^o, z>y$ such
  that $\tilde{Κ}^+(z)>T$. Without loss, we can assume that
  $\tilde{Κ}(z)>T$, too.

  If there exists an $x ∈ e^o$, $x<y$ such that $\tilde{Κ}^+(x)<T$,
  then by similar reasons, the closed interval $l = \set{ t ∈ e^o \mid x
    ≤ t ≤ z }$ satisfies~\ineqref{eq:wk-good-interval}.

  To complete the proof, assume there does not exist an $x ∈ e^o$,
  $x<y$ such that $\tilde{Κ}^+(x)<T$. Then $\tilde{Κ}^+$ is constant
  on the set $e^o ∩ \set{ x ≤ y }$. Since $\tilde{Κ}^+$ is continuous,
  this implies that $T = \tilde{Κ}^+(\min e)$, too. If
  $\tilde{Κ}^+(\min e) = \tilde{Κ}(\min e)$, then $T$ is in the image
  of the vertex set of $Γ_H$. Absurd. Therefore,
  $T=\tilde{Κ}^+(\min e) > \tilde{Κ}(\min e)$ and by the continuity of
  $\tilde{Κ}|e$, there is an $x ∈ e^o$ such that the closed interval
  $l = \set{ t ∈ e^o \mid x ≤ t ≤ z }$
  satisfies~\ineqref{eq:wk-good-interval}.
\end{proof}

\begin{remark}
  \label{rem:wk-mechanical-and-others}
  Let us elaborate on claim~\ref{cl:wk-mech-implies-k-surjective}. If
  $H$ is mechanical, $p = H_p$, so
  $\tilde{Κ}_k = (k+2)^{-1} × f_{k+2} / f_k$. In addition, the change
  of variables in \eqref{eq:nh-mean-ke} implies that for $k ≥ 3$,
  $2 π f_k(I) = k(k-2) H_I(I) × \iint p^{k-3} \D{p} ∧ \D{q}$. These
  two facts imply that $\tilde{Κ}_1$ tends to $0$ as $I$ approaches a
  critical action, while $\tilde{Κ}_k$ tends to a non-zero limit for
  $k ≥ 3$ when the critical action is positive. Figure~\ref{fig:wk-k}
  plots the graphs of $\tilde{Κ}_k$ for selected values of $k$ and
  demonstrates these facts for a selected example. In addition, when
  $H = ½ \left( p ² + q ² \right)$ is a simple harmonic oscillator,
  $\tilde{Κ}_k = \frac{2}{k+1} I$. This fact is seen in
  figure~\ref{fig:wk-k}, too.

  Figure~\ref{fig:wk-a-neighbourhood-of-a-critical-cycle} depicts the
  level sets of a planar hamiltonian with a saddle cycle $v$ with
  incoming edges $A$, $B_0$ and $B_1$ and an outgoing edge $E$ (and
  there are unique minima inside each white disk-shape
  region). Case~\ref{it:cl:wk-mech-implies-k-surjective-3} applies to
  the maximum of both $A$ and $B_0$, so the graph of $\tilde{Κ}$ on
  each edge should be roughly
  $∩$-shaped. Case~\ref{it:cl:wk-mech-implies-k-surjective-2} applies
  to the maximum of $B_1$ and one expects the graph of $\tilde{Κ}$ to
  be roughly
  $/$-shaped. Case~\ref{it:cl:wk-mech-implies-k-surjective-3} also
  applies to the minimum vertex of $E$, so the graph of $\tilde{Κ}|_E$
  should increase monotonically from $0$.
\end{remark}

It remains to prove that the vector field $\bar{R}_0$ does not have a
constant period. The hamiltonian $\mean{G}_{0,β}$ of the averaged
vector field $\mean{R}_0$ and the latter's hamiltonian
$\mean{G}_{0,T}$ with respect to the symplectic form
$-T ω_{β}/\mean{G}_{0,β}$ from Lemma~\ref{lem:ave-transformation} are
computed to be
\begin{align}
  \label{eq:wk-g0b}
  \mean{G}_{0,β} &= (β Z(β) H_I(I))^{-1}\, f(I) \, \tilde{Κ}(I) \, ζ_l(ξ) \,  × \exp(-β G_{β}(I,ξ)) \\
  \label{eq:wk-r0}
  \mean{R}_0 &= \dfrac{f(I) \, \tilde{Κ}(I)}{H_I(I)^2} × \left( -ξ^l \, ∂_I + ζ_l(ξ) \, H_I(I) \, \left[ \dfrac{\tilde{Κ}(I) - T}{\tilde{Κ}(I)}  \right] \, ∂_{ξ} \right) \\
  \label{eq:wk-g0}
  \mean{G}_{0,T} &= Ζ_l(ξ) + H(I) - T \ln F_k(I),
\end{align}
where $Ζ_l(ξ) = ½ ξ ² - T \ln (ζ_l(ξ)/ζ_l(0))$ and
$F=F_k$ is defined in~\eqref{eq:wk-fk}. One checks that $Ζ_1(ξ) = ½ ξ ²$
and $F_1(I) = I$, so \eqref{eq:wk-g0} reproduces the hamiltonian in
\eqref{eq:nh-g0}.

Let it be noted that $(I,ξ)$ are not Darboux coordinates for the
rescaled symplectic form. One defines Darboux coordinates via
\begin{align}
  \label{eq:wk-darboux-coordinates}
  σ &= \int_{I_0}^I \frac{(H_I)^2}{f(I)\, \tilde{Κ}(I)} \, \D{I} = \int_{H_0}^H \D{H}/F, &&& χ &= \int_0^{ξ} \frac{\D{ξ}}{ζ_l(ξ)}
\end{align}
which for $k=1, l=1$ specializes to the Darboux coordinates introduced
above for the Hoover thermostat.
\begin{lemma}
  \label{lem:wk-kam-nec}
  Let $H : \cotangent Σ → \R$ be a well-behaved, smooth
  hamiltonian. Then for all $T ∈ τ_H$, the Watanabe-Kobayashi
  thermostat~\eqref{eq:wat-kob-separable} is a separable thermostat in
  the sense of definition~\ref{def:separable-thermostat}. In addition,
  if $l>1$, then
  conditions~\ref{thm:separable-thermostat-cond1}--\ref{thm:separable-thermostat-cond3}
  of Theorem~\ref{thm:separable-thermostat} are satisfied.
\end{lemma}
\begin{proof}
  Lemma~\ref{lem:wk-condition-2-is-almost-always-satisfied} proved
  that condition~\ref{cond:2} of the definition of a separable
  thermostat is satisfied when $T ∈ τ_H$. It is clear that the other
  conditions of definition~\ref{def:separable-thermostat} hold on the
  same interval from that lemma. 

  Let now $T ∈ τ_H$ and let $H_0$ (resp. $I_0$) be the energy
  (resp. action) at the thermostatic equilibrium on the
  above-mentioned
  interval. Condition~\ref{thm:separable-thermostat-cond1} holds by
  inspection of~\eqref{eq:wat-kob-separable}, and
  condition~\ref{thm:separable-thermostat-cond2} holds in view of
  lemma~\ref{lem:wk-condition-2-is-almost-always-satisfied}. It
  remains to verify the non-constancy of the period of $\mean{R}_0$.

  The proof of
  lemma~\ref{lem:wk-condition-2-is-almost-always-satisfied} shows that
  $\tilde{Κ}-T$ is negative (resp. positive) on an interval to the
  left (resp. right) of $I_0$. This implies that near $I_0$ the graph
  of the ``potential'' $U_T(I)$ is $\cup$-shaped modulo possible bad
  behaviour near in a neighbourhood of $I_0$ itself. This implies that
  $\mean{G}_{0,T}$ is proper in some neighbourhood of $(I_0,0)$.

  On the other hand, the hessian of $\mean{G}_{0,T}$
  \begin{equation}
    \label{eq:wk-g0t-hessian}
    \hessian \mean{G}_{0,T} = ζ_l(ξ)^{-1} \, ξ^{l-1} (l - β ξ^{l+1}) \, \left( \D ξ \right) ² + T \, \tilde{Κ}^{-2} \, \tilde{Κ}' \, \left( \D H \right) ².
  \end{equation}
  If $\tilde{Κ}' ≠ 0$ at $I_0$, then the hessian is positive definite
  at the thermostatic equilibrium when $l=1$; otherwise it is
  degenerate. The lowest order term in the coefficient on $(\D ξ)^2$
  is of degree $l-1$ in $ξ$ since $ζ_l(0)=(2T)^n \, n!$. This implies
  that for $l>1$,
  $\mean{G}_{0,T} = ((2T)^n \, n!)^{-1}\, (l+1)^{-1}\, ξ^{l+1} + ½
  T^{-1}\,\tilde{Κ}'(H_0)\, (H-H_0)^2 + O(ξ^{l+3},(H-H_0)^3)$. Since
  $l$ is odd, $\mean{G}_{0,T}$ is proper in a neighbourhood of its
  critical point. If $l>1$, then the degeneracy of the hessian of
  $\mean{G}_{0,T}$ at the critical point implies that $\mean{R}_0$ is
  not conjugate to its linearization at the critical point, and
  therefore its period is not constant.

  If $\tilde{Κ}' = 0$ at $I_0$, the hessian is degenerate and the
  period of $\mean{R}_0$ cannot be constant in a neighbourhood of
  $(I_0,0)$ for similar reasons.
  
  This proves the second part of the lemma.
\end{proof}

One is now in a position to prove the analogue to
Theorems~\ref{thm:nh-kam-nec} \& \ref{thm:logistic-nec-kam}.

\begin{theorem}
  \label{thm:wk-nec-kam}
  Let $H : \cotangent Σ → \R$ be a well-behaved, real-analytic
  hamiltonian. Then, for all $T ∈ τ_H$, the period of $\bar{R}_0$ is
  not constant in a neighbourhood of the critical point $(I_0,0)$.
\end{theorem}

\begin{proof}
  All cases except $l=1$ and $k ≥ 3$ follow from
  Lemma~\ref{lem:wk-kam-nec} or Theorem~\ref{thm:nh-kam-nec}, so
  assume that $l=1$ and $k ≥ 3$.

  As in the proof of Theorem~\ref{thm:nh-kam-nec}, let $γ_0 ∈ Γ_H$ be
  the connected component of the level set of $H$ where the
  thermostatic equilibrium $(I_0,0)$ is attained and let $c_0 ⊂ Γ_H$
  be the edge containing $γ_0$.

  If $c_0$ has a saddle vertex, then either
  case~\ref{it:cl:wk-mech-implies-k-surjective-2}
  or~\ref{it:cl:wk-mech-implies-k-surjective-3} of
  claim~\ref{cl:wk-mech-implies-k-surjective} applies. In both cases,
  $H$ and $F$ are bounded in a $c_0$-neighbourhood of the vertex and
  hence the potential $U_T=H-T \ln F$~\eqref{eq:wk-g0} is similarly
  bounded. This contradicts isochronicity. Thus, if $\bar{R}_0$ is
  isochronous, then $c_0$ has no saddle vertex. Hence
  claim~\ref{cl:1-no-crit-pts} holds.

  To prove claim~\ref{cl:2-consequences} holds here, one computes, in the Darboux
  coordinates $(σ,ξ)$~\eqref{eq:wk-darboux-coordinates}, that
  \begin{equation}
    \label{eq:wk-u_s}
    U_{σ} = I_H \, f \, \left( \tilde{Κ} - T \right) = F_H \, \left( \tilde{Κ} - T \right).
  \end{equation}
  Assume that the local minimum vertex $\min c_0$ is
  non-degenerate. Then $I_H → ω^{-2} > 0$ as $γ ↘ \min c_0$. On the
  other hand, both $f$ and $\tilde{Κ}$ approach $0$ as $γ ↘ \min
  c_0$. This implies that $U_{σ} ↗ 0$ and so $U_{σ}$ is bounded on
  $(\min c_0, γ_0] ⊂ c_0$. By the same argument as in the second
  paragraph following claim~\ref{cl:2-consequences}, one obtains a
  contradiction with isochronicity, thereby proving that claim
  here. Therefore, the local minimum vertex cannot be non-degenerate
  if the vector field $\bar{R}_0$ is isochronous. This completes the
  proof.
\end{proof}

\begin{figure}[h]
  \caption{$\tilde{Κ}_k$, rescaled by $\frac{k+1}{2}$, for the planar pendulum $H = ½ p ² - \cos q$ with $k=1,3,5,7$. The inset (upper left) highlights the behaviour near the critical energy level $H=1$.}
  \label{fig:wk-k}
  \includegraphics[width=10cm,keepaspectratio]{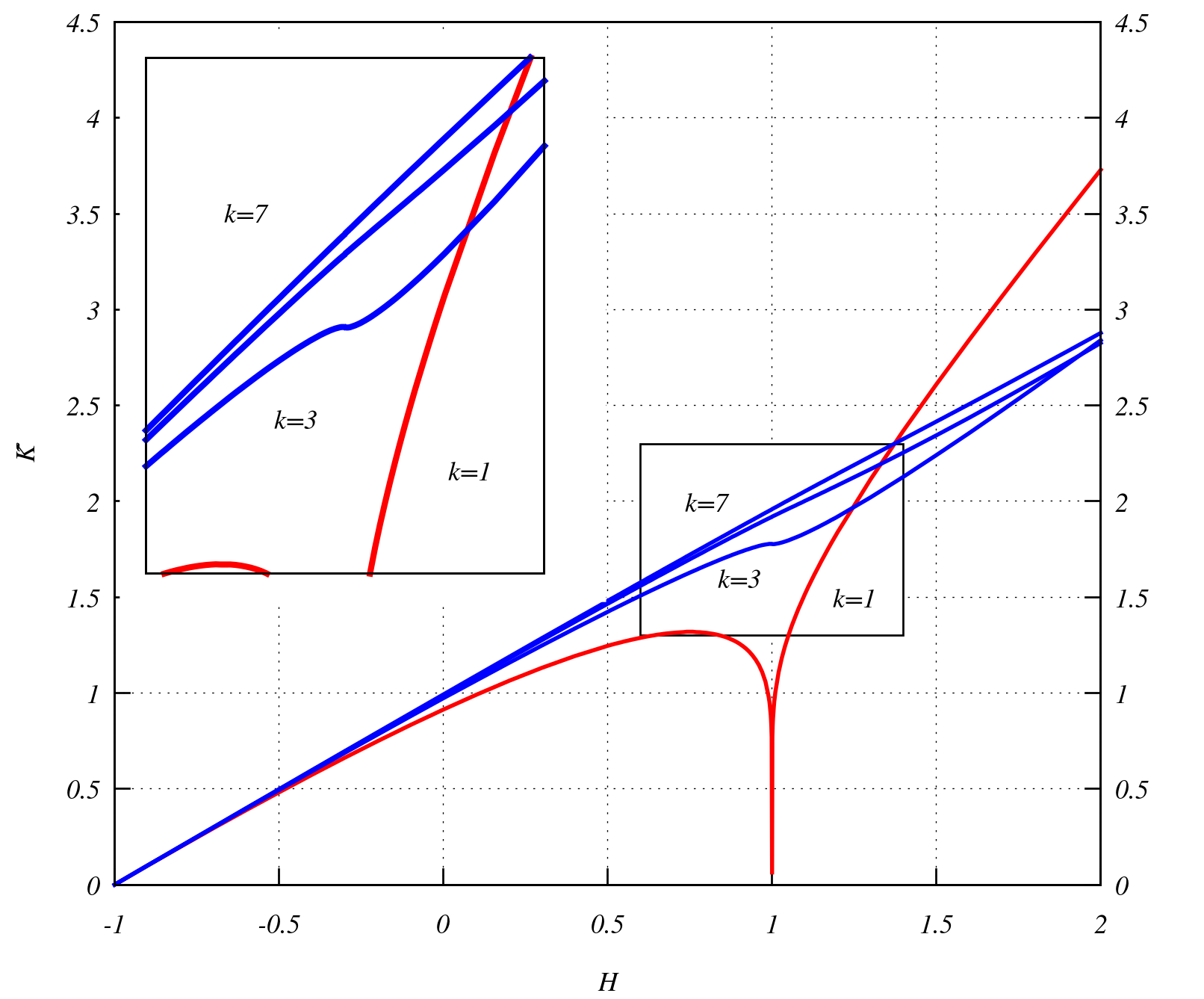} 
\end{figure}

\begin{figure}[h]
  \caption{$\ln(f_k)$ vs. $\ln(H+1)$ for the planar pendulum $H = ½ p ² - \cos q$ with $k=3,5,7,9$. The inset (upper left) highlights the behaviour near crossings. The approximate linearity to the left of $H=0$ is striking.}
  \label{fig:wk-lnf-lnh}
  \includegraphics[width=10cm,keepaspectratio]{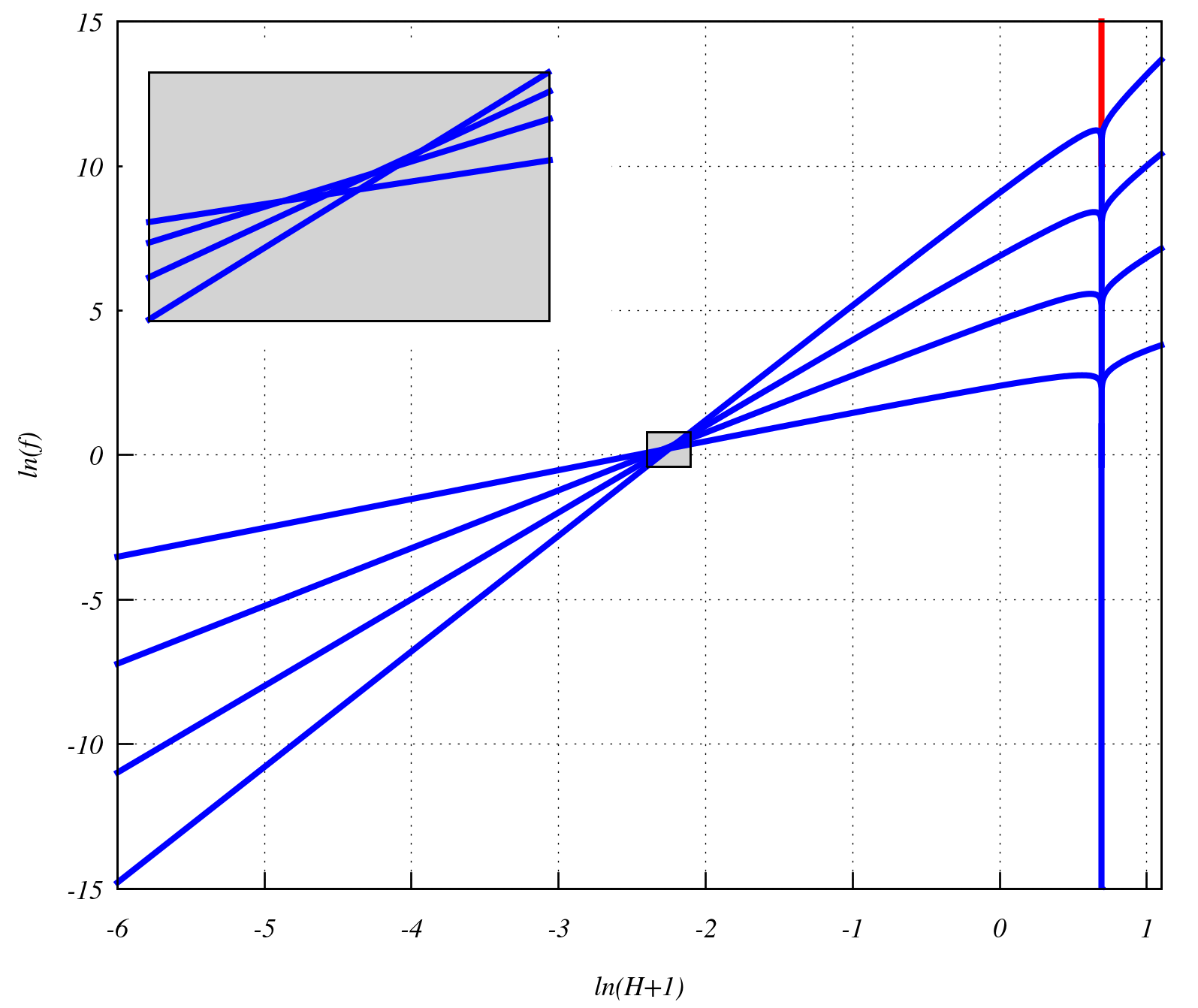} 
\end{figure}

\subsection{Hoover \& Sprott and Hoover, Sprott \& Hoover}
\label{sec:hs-hsh}

In \cite{10.1016/j.cnsns.2015.08.02}, Hoover, Sprott and Hoover obtain
numerical results that indicate for some parameter values there are
large sets with positive Lyapunov exponents for the thermostat with

\begin{align}
  \label{eq:hsh-sh-separable}
  G_{β} &= H + ¼ ξ ⁴, &&& A_0 &= q, & A_1 &= -ξ ³,\\
  \notag
  B_0 &= p ³, & B_1 &= -μ ξ ³, & C_1 &= 1, & C_0 &= \left[ q H_q - T \right] + μ p ² \left[ p  H_p - 3T \right].
\end{align}
For comparison
with~\cite[eq. {[}HS{]},p. 237]{10.1016/j.cnsns.2015.08.02}, their
$ξ⁴$ term is not multiplied by $β$ and their $T$ (resp. $β$, $α/T$) is
$T$ (resp. $ε$, $ε μ$) here. It is trivial to see that conditions
\ref{cond:1}, \ref{cond:3} and \ref{cond:4} of
definition~\ref{def:separable-thermostat} are satisfied. To prove that
condition \ref{cond:2} holds, note that on averaging $C_0$, one
obtains
\begin{equation}
  \label{eq:hsh-ave}
  \mean{C}_0 = \left[ Κ(I) - T \right] + μ f(I) \left[ \tilde{Κ}(I) - T \right],
\end{equation}
where $Κ$ is the averaged temperature from \eqref{eq:nh-mean-ke},
$\tilde{Κ}=\tilde{Κ}_3$ is the weighted average temperature from
\eqref{eq:wk-c0} with $k=3$ and $f=f_3$ is defined likewise.

The hamiltonian $\mean{G}_{0,β}$ of the averaged vector field
$\mean{R}_0$ and the latter's hamiltonian $\mean{G}_{0,T}$ with
respect to the symplectic form $-Tω_{β}/\mean{G}_{0,β}$ from
Lemma~\ref{lem:ave-transformation} are computed to be
\begin{align}
  \label{eq:hsh-g0b}
  \mean{G}_{0,β} &= (β Z(β) H_I(I))^{-1}\, \left( Κ(I) + μ f(I) \, \tilde{Κ} \right) \exp(-β G_{β}(I,ξ)) \\
  \label{eq:hsh-r0}
  \mean{R}_0 &= \frac{1}{H_I} × \left( -ξ³ \, (H_I)^{-1} \, \left[ Κ(I)+μ f(I) \, \tilde{Κ}(I) \right]\, ∂_I + \mean{C}_0 \, ∂_{ξ} \right) \\
  \label{eq:hsh-g0}
  \mean{G}_{0,T} &= ξ⁴/4 + \underbrace{H(I) - T \ln Q_{μ}(I)}_{U_T},
\end{align}
where
$\ln Q_{μ} = \int^H_{H_0} \D{H} \, \left[ \frac{1+μ f}{Κ + μ f
    \tilde{Κ}} \right]$. One checks that for $μ = 0$ \eqref{eq:hsh-g0}
reproduces the hamiltonian in \eqref{eq:nh-g0}, while for $μ = ∞$, the
potential coincides with that in~\eqref{eq:wk-g0}.

\begin{lemma}
  \label{lem:hsh-kam-nec}
  Let $H : \cotangent Σ → \R$ be a well-behaved, smooth
  hamiltonian. Then, for all $T, μ > 0$, the thermostat
  $Τ$~\eqref{eq:hsh-sh-separable} is a separable thermostat in the
  sense of definition~\ref{def:separable-thermostat}.
\end{lemma}
The proof is similar to that of~\ref{lem:nh-is-a-separable-thermostat}.

The following is an immediate consequence of the preceding lemma and
the fact that $\mean{G}_{0,T}$ is quartic in $ξ$.

\begin{theorem}
  \label{thm:hsh-nec-kam}
  Let $H : \cotangent Σ → \R$ be a well-behaved, smooth
  hamiltonian. Then the period of $\bar{R}_0$ is not constant in a
  neighbourhood of the critical point $(I_0,0)$.
\end{theorem}

\bibliographystyle{abbrv}
\bibliography{nhwc-references}
\end{document}